\documentclass[pdflatex,sn-mathphys-num]{sn-jnl}

\usepackage{amsmath,amssymb,amsthm,graphicx,subfigure,float,caption,epstopdf,tabularx,color, bm,amsfonts,epic}
\usepackage{diagbox}
\usepackage{graphicx}%
\usepackage{multirow}%
\usepackage{amsmath,amssymb,amsfonts}%
\usepackage{amsthm}%
\usepackage{mathrsfs}%
\usepackage[title]{appendix}%
\usepackage{xcolor}%
\usepackage{textcomp}%
\usepackage{manyfoot}%
\usepackage{booktabs}%
\usepackage{algorithm}%
\usepackage{algorithmicx}%
\usepackage{algpseudocode}%
\usepackage{listings}%

\newtheorem{thm}{Theorem}[section]

\newtheorem{rmk}{Remark}[section]

\newtheorem{ex}{Example}[section]
\newtheorem*{prf}{Proof}
\numberwithin{equation}{section}
\renewcommand{\figurename}{\textbf{Fig.}}
\renewcommand{\tablename}{\textbf{Table.}}

\newcommand{\bma}{{\bm \alpha}}
\numberwithin{equation}{section}
\renewcommand{\figurename}{\textbf{Fig.}}
\renewcommand{\tablename}{\textbf{Table.}}
\theoremstyle{thmstyleone}%
\theoremstyle{thmstyletwo}%

\theoremstyle{thmstylethree}%

\begin{document}

\title[Article Title]{High-order multi-structures-preserving exponential integrators for the derivative nonlinear Schr\"{o}dinger equation}


\author[1]{\sur{Liping Wu}}\email{lipingwu1981@163.com}

\author[2]{\sur{Li Yang}}\email{hnynyl@163.com}

\author*[2]{\sur{Chaolong Jiang}}\email{Chaolong Jiang}

\affil[1]{\orgdiv{Faculty of Information Engineering and Automation}, \orgname{Kunming University of Science and Technology}, \orgaddress{\city{Kunming}, \postcode{ 650500}, \country{China}}}

\affil*[2]{\orgdiv{School of Statistics and Mathematics}, \orgname{Yunnan University of Finance and Economics}, \orgaddress{\city{Kunming}, \postcode{650221},  \country{China}}}



\abstract{This paper presents a novel class of high-order mass-, energy- and momentum-preserving exponential integrators for solving the derivative nonlinear Schr\"{o}dinger equation. Firstly, we reformulate the original system into an exponential supplementary variable system based on the idea of the exponential supplementary variable approach, and then the reformulated system is discretized by using the standard Fourier pseudo-spectral method in space and the high-order prediction and correction Lawson Runge-Kutta method in time, respectively. The proposed method is highly efficient, temporally high-order accurate, and simultaneously preserves the mass, energy and momentum in the discrete setting. Finally, numerical experiments validate the accuracy and energy-preserving properties.}




\maketitle

\section{Introduction}\label{sec1}
In this paper, we consider the derivative nonlinear Schr\"odinger (DNLS) equation\cite{CLP-Geo-1999,MOMT-JPSJ-1976,Mjolhus-JPP-1976} taking on the form
\begin{equation}
\label{DNLS-equation}
\left\{ \aligned
 &\text{i}\frac{\partial}{\partial t}\psi(x,t)=-\partial_{xx}\psi(x,t)-{\rm i}\partial_x\big(|\psi(x,t)|^2\psi(x,t)\big),\quad  t>0,\\
&\psi(x,0)=\psi_0(x), \quad \quad x \in \Omega,
\endaligned  \right.
\end{equation}
where $t$ is the time variable, $x$ is the spatial variable, $\text{i}=\sqrt{-1}$ is the complex unit, $\psi:=\psi({x},t)$ is the complex-valued wave function, and $\Omega=[a,b]\subset\mathbb{R}$ is a bounded domain with a periodic boundary condition. The DNLS equation \eqref{DNLS-equation} conserves the mass, energy and momentum as follows, \cite{BP-2022-InM,KN-JMP-1978}
\begin{align}\label{DNLS-mass}
&\mathcal{M}[\psi(\cdot,t)]=\int_a^b|\psi|^2 {\rm d}{x}\equiv\mathcal{M}[\psi(\cdot,0)],\\\label{DNLS-H-energy}
&\mathcal{E}[\psi(\cdot,t)] = \int_a^b\left(|\partial_x\psi|^2-\frac{3}{2}\Im(|\psi|^2\psi\partial_x(\psi^\ast))+\frac{1}{2}|\psi|^6\right){\rm d}{x}\equiv\mathcal{E}[\psi(\cdot,0)],\\
\label{DNLS-momentum}
&\mathcal{P}[\psi(\cdot,t)]=\int_a^b\left(\Im(\psi^\ast\psi_x)+\frac{1}{2}|\psi|^4\right) {\rm d}{x}\equiv\mathcal{P}[\psi(\cdot,0)],
\end{align}
where $\psi^\ast$ is the complex conjugation of the $\psi$, and $\Im (f)$ represents the imaginary part of $f$.

Extensive mathematical and numerical studies have been conducted for the DNLS \eqref{DNLS-equation} in the literature. Along the mathematical front, for the derivation of  well-posedness and existence of the solution of the DNLS \eqref{DNLS-equation}, we refer to \cite{BP-2022-InM,HO-PD-1992,PSS-DPDE-2017,Ozawa-IUMJ-1996} and the references therein. Along the numerical front, different efficient and accurate numerical methods including the predictor-corrector method \cite{GP-PF-1987}, the Ablowitz-Ladik method \cite{DF-JCP-1988}, the variational method \cite{HS-PS-2009} and the physics-informed neural network method \cite{PPC-WM-2021,PLC-ND-2021} have been developed for the DNLS \eqref{DNLS-equation}. However, these efficient numerical schemes usually cannot exactly  conserve the conservation laws of the DNLS \eqref{DNLS-equation}.

In \cite{LL-1995-SNA}, Li and Vu-Quoc pointed out:``\emph{... in some areas, the ability to preserve some invariant properties of the original differential equation is a criterion to judge the success of a numerical simulation.}" Thus, it is crucial to develop numerical schemes that conserve discretized versions of conservation laws of the continuous dynamical systems, and they are usually called energy-preserving or conservative schemes \cite{ELW06}. In \cite{Fla-JCP-1992}, an implicit numerical scheme that conserves the mass \eqref{DNLS-mass} is proposed. In \cite{HWZ-MCS-2024,Tsuchida-2002}, the integrable discretizations of derivative nonlinear
Schr\"odinger equations are investigated. Later on, Li et al. \cite{LLS-IJNSNS-2018,LLCL-IJMSSC-2017} presented two classes of mass-conserving Crank-Nicolson finite difference schemes.
Guo et al. \cite{GF-JAAC2021} introduced and analyzed a Crank-Nicolson finite difference scheme for the DNLS equation with the Riesz space fractional derivative. It can be proven rigorously in mathematics that the Crank-Nicolson scheme (CNS) can satisfy the discretized version of the mass \eqref{DNLS-mass}, and is second-order accuracy in time and space. Recently, Xue et al. \cite{XZ-pd-2024} proposed four classes of mass-conserving finite difference methods. However, all of existing energy-preserving schemes only have second-order accuracy in time, which may lose accuracy in practical computation for a given large time step. Additionally, they usually can only  conserve a single conservation law. Based on the basic principle where the numerical method should preserve the intrinsic properties of the original problems as much as possible\cite{FQ-shs-2003}, it is desirable to develop numerical methods that  exactly conserve multiple geometric/invariants properties of the system, called multi-structures-preserving (mSP) methods.

In the past few decades, many attempts to develop mSP schemes have been done on conservative systems. In \cite{CQT-CWM-2002,Reich-JCP-2000}, the Gauss collocation method \cite{ELW06} has been shown to be powerful in symplecticity, mass and symmetry preservation of the NLS-type equations. Nevertheless, the symplectic schemes usually cannot conserve the Hamiltonian energy of the NLS-type equations. In \cite{MQR-PTRSA-1999}, McLachlan et al. presented a procedure for preserving multiple invariants by a numerical integrator for ordinary differential equations (ODEs). The basic idea of their method is that the antisymmetric matrix of the dynamic system is replaced by an antisymmetric tensor taking discrete gradients of all invariants to be preserved as input. Calvo et al. \cite{CHMR06} proposed a directional projection method to preserve multiple invariants. Dahlby et al. \cite{DOY-JPA-2011} employed discrete gradients as the underlying tool to present a new methodology for preserving multiple invariants in ODEs. Based on the idea of the line integral method, Brugnano et al. \cite{BS-NA-2014} presented the concept of the enhanced line integral method for Hamiltonian ODEs that allows for the conservation of multiple invariants of the dynamical system. In recent years, the idea of the relaxation time stepping method is generalized and the concept of the multiple-relaxation method is presented in \cite{BK-JSC-2023,LL-SJNM-2024,LLZZ-JCP-2025,LZ-CPC-2026}, which is shown to be an efficient method for constructing multiple invariants-conserving schemes for the conservative systems. We note that these methods are presented and analysed on the ODE framework, which implies when generalizing these methods from conservative ODEs to conservative partial differential equations (PDEs), one first shall discretize PDEs in space to yield an ODE system with multiple invariants. However, as pointed out in \cite{LL-SJNM-2024}, it is generally difficult to preserve the mass, energy and momentum of the NLS-type equation simultaneously after the spatial discretization. Thus, how to construct mass-, energy- and momentum-conserving schemes for solving the DNLS \eqref{DNLS-equation} is challenging, which motivates this paper.

 In this paper, based on the idea of the exponential supplementary variable method (ESVM) \cite{GJZ2024} viewed as a promotion of the idea of the classical SVM method \cite{GHWCMAME2021,HWGCMAME2023}, we present a novel class of high-order multi-structures-preserving exponential integrators for solving the DNLS equation \eqref{DNLS-equation}. The basic idea
behind the ESVM method is first to reformulate the original system together with the mass, energy and momentum into a constraint-free system by introducing three additional supplementary variables. Then, the Lawson transformation \cite{BDL-2017-SUM,HO10an,JCQS-JSC-2022,Lawson1967} is employed to exactly integrate the linear term of the reformulated system, and an ESVM system is presented. Finally, we discretize the ESVM system using the standard Fourier pseudo-spectral method \cite{CQ-ETNA-2001,ST06} and high-order prediction-correction Lawson Runge-Kutta method for spatial and temporal variables, respectively. The proposed method is highly efficient, temporally high-order accurate, and can exactly preserve the mass, energy and momentum simultaneously. To our knowledge, this is the first temporally high-order multi-structures-preserving method for the DNLS equation \eqref{DNLS-equation}.

The structure of the paper is arranged as follows. In Section \ref{Sec:ESVM-DNLS:2}, based on the idea of the ESVM reformulation of the DNLS equation \eqref{DNLS-equation}, a class of fully-discrete high-order multi-structures-preserving exponential integrators is developed. Some theoretical analysis on the proposed semi-discrete scheme is conducted in Section \ref{DNLS-STA}. In Section \ref{Sec:ESVM-DNLS:4}, extensive numerical results are addressed to confirm the accuracy and energy-preserving properties of the new scheme. Finally, we conclude in Section \ref{Sec:ESVM-DNLS:5}.

\section{Numerical algorithms}\label{Sec:ESVM-DNLS:2}
 In this section, we present a class of mSP exponential integrators for computing the DNLS \eqref{DNLS-equation} based on the idea of the ESVM approach \cite{GJZ2024} and the standard Fourier pseudo-spectral method \cite{CQ-ETNA-2001,ST06}. The proposed scheme can conserve the mass, energy and momentum \eqref{DNLS-mass}-\eqref{DNLS-momentum} simultaneously, and achieves high-order accuracy in spatial and temporal directions, respectively.

 \subsection{Temporal discretization}\label{Sec:ESVM-DNLS:3}

 In this subsection, we employ the ESVM approach and the prediction-correction Lawson RK method to discretize the temporal
variable of the DNLS \eqref{DNLS-equation} and a class of semi-discrete mSP exponential integrators is proposed. To begin with, we rewrite \eqref{DNLS-equation} as
\begin{align}\label{DNLS-equi}
\partial_t \psi = \mathcal L\psi+\mathcal N[\psi],\quad \mathcal L={\rm i}\partial_{xx},\quad \mathcal N[\psi]=-\partial_x\big(|\psi|^2\psi\big).
\end{align}
Since the system \eqref{DNLS-equi} satisfies the energy, mass and momentum \eqref{DNLS-mass}-\eqref{DNLS-momentum}, we have
\begin{align}\label{DNLS-conservation1}
 \frac{d}{dt}\mathcal{E}=2\Re\Big(\frac{\delta \mathcal{E}}{\delta {\psi}^\ast},\mathcal L\psi+\mathcal N[\psi]\Big)=0,\quad  \frac{d}{dt}\mathcal{M}=2\Re\Big(\frac{\delta \mathcal{M}}{\delta {\psi}^\ast},\mathcal L\psi+\mathcal N[\psi]\Big)=0,
\end{align}
  and
 \begin{align}\label{DNLS-conservation2}
  \frac{d}{dt}\mathcal{P}=2\Re\Big(\frac{\delta \mathcal{P}}{\delta {\psi}^\ast},\mathcal L\psi+\mathcal N[\psi]\Big)=0,
  \end{align}
where, $\Re (f)$ represents the real part of $f$, $(f,g) = \int_a^b f(x)g(x)^\ast \mathrm{d}x$ denotes the inner product of $f(x),~g(x)$ in $L^2([a,b])$, and the corresponding norm in $L^2([a,b])$ is defined as $\| f\|^2 = (f,f)$.

 Then, let $t_{n} = n \tau$ and $t_{ni} = t_{n} + c_{i} \tau, n = 0, 1,\cdots,$ where $\tau $ be the time step, and denote $\psi^{n}$ and $\psi_{ni}$ as the numerical approximations of the function $\psi(x,t_{n})$ at $t_{n}$ and $t_{ni}$, respectively. We recall some common and useful results on the semi-discrete Lawson RK method and the prediction-correction Lawson Runge-Kutta (RK) method for system \eqref{DNLS-equi}, respectively. Assume $b_i,a_{ij}\ (i,j=1,\cdots,s)$ be real numbers and $c_i=\sum_{j=1}^sa_{ij}$, and an $s$-stage Lawson RK method applied to \eqref{DNLS-equi} is given by \cite{HO10an}
		\begin{equation}\label{shm:lawson-RK-method}
			\left\{
			\begin{aligned}
				&\psi_{ni} = \exp(c_i\tau\mathcal{L}) \psi^n + \tau\sum_{j=1}^{s}a_{ij}\exp\left((c_i-c_j)\tau\mathcal{L}\right) \mathcal N[\psi_{nj}],\\
				&\psi^{n+1}=\exp(\tau\mathcal{L}) \psi^n + \tau\sum_{i=1}^{s}b_i \exp((1-c_i)\tau\mathcal{L}) \mathcal N[\psi_{ni}].
			\end{aligned}
			\right.
		\end{equation}
	In general, the coefficients are displayed as follows:
	\begin{table}[h]
\begin{tabular}{c|ccc|c}
			${c_1}$ & $a_{11}$& $\cdots$&$a_{1s}\exp\big((c_1-c_s)\tau\mathcal{L}\big)$&$\exp(c_1\tau\mathcal{L})$\\
			${\vdots}$ &$\vdots$& $\ddots$&$\vdots$&$\vdots$\\
			${c_s}$ & $a_{s1}\exp\big((c_s-c_1)\tau\mathcal{L}\big)$& $\cdots$&$a_{ss}$&$\exp(c_s\tau\mathcal{L})$\\
			\hline
			& $b_1\exp\big((1-c_1)\tau\mathcal{L}\big)$ &$\cdots$&$b_s\exp\big((1-c_s)\tau\mathcal{L}\big)$&$\exp(\tau\mathcal{L})$\\
		\end{tabular}
		\caption{The coefficients of the Lawson RK method \eqref{shm:lawson-RK-method}.}\label{tab-4-6-EIM}
	\end{table}
	
	As stated in \cite{HO10an}, the above Lawson Runge-Kutta method \eqref{shm:lawson-RK-method} will be reduced to a classical RK method if $\mathcal L=0$, and
	the classical method will be henceforth referred to as the underlying Runge-Kutta method. Additionally, if the Lawson RK method \eqref{shm:lawson-RK-method} is implicit, one needs to solve a large fully nonlinear system at every time step, and thus it may be very time consuming in practical computations. To address this problem, we present the prediction-correction Lawson RK method. Let $b_i,a_{ij}\ (i,j=1,\cdots,s)$ be real numbers and $c_i=\sum_{j=1}^sa_{ij}$, and an $s$-stage prediction-correction Lawson RK method applied to \eqref{DNLS-equi} is defined as	
		\begin{itemize}
			\item    {\bf Prediction:} let $\psi_{ni}^{(0)} = \exp(c_i \tau \mathcal{L})\psi^n$ and $M$ be a given positive integer, we iteratively compute $\psi_{ni}^{(\ell+1)}$ from $\ell=0$ to $M-1$ as follows:
			\begin{equation}\label{eq:prediction-scheme}
				\psi_{ni}^{(\ell+1)} = \exp(c_i\tau\mathcal{L}) \psi^n + \tau\sum_{j=1}^{s}a_{ij}\exp\left((c_i-c_j)\tau\mathcal{L}\right) \mathcal{N}[\psi_{nj}^{(\ell)}].
			\end{equation}
			\item{\bf Correction:} update $\psi^{n+1,(M)}$ by
			\begin{equation}\label{eq:correction-scheme}
				\psi^{n+1,(M)}=\exp(\tau\mathcal{L}) \psi^n + \tau\sum_{i=1}^{s}b_i \exp((1-c_i)\tau  \mathcal{L})~\mathcal{N}[\psi_{ni}^{(M)}].
			\end{equation}
		\end{itemize}
The convergence order in local truncation error of the prediction-correction Lawson RK method has been carefully analyzed in \cite{GJZ2024}. However, we note that the prediction-correction prediction-correction Lawson RK method cannot conserve the mass, energy and momentum simultaneously of \eqref{DNLS-equation}. Below, we employ the idea of the ESVM approach and the prediction-correction Lawson RK method to construct a class of mass-, energy- and momentum-preserving  exponential integrators for computing the DNLS \eqref{DNLS-equation}.

  To this end, we first introduce three supplementary
variables $\beta_1$, $\beta_2$ and $\beta_3$ to incorporate the mass, energy and momentum \eqref{DNLS-mass}-\eqref{DNLS-momentum} into \eqref{DNLS-equi}, and have
 \begin{equation}\label{DNLS-SVM-Formulation}
\left\lbrace
\begin{aligned}
&\partial_t \psi = \mathcal L\psi+\mathcal N[\psi]+\beta_1 g_1[\psi]+\beta_2g_2[\psi]+\beta_3 g_3[\psi],\\\
&\frac{{\rm d}}{{\rm d}t}\mathcal E[\psi]=0,\quad\frac{{\rm d}}{{\rm d}t}\mathcal M[\psi]=0,\quad  \frac{{\rm d}}{{\rm d}t}\mathcal P[\psi]=0.
\end{aligned}\right.
  \end{equation}
where $g_1[\psi],\ g_1[\psi]$ and $g_3[\psi]$ are user-supplied functionals or functionals of $\psi$ given by
\begin{equation}\label{dNLS-projection-direction}
g_1[\psi]= \frac{\delta \mathcal E}{\delta {\psi}^\ast}=
              -\partial_{xx}\psi-3{\rm i}|\psi|^2\partial_x\psi+\frac{3}{2}|\psi|^4\psi,\ g_2[\psi] = \frac{\delta \mathcal M}{\delta {\psi}^*}=\psi,\  g_3[\psi]=\frac{\delta \mathcal P}{\delta {\psi}^\ast}=-{\rm i}\psi_x+|\psi|^2\psi.
\end{equation}
Then we employ the Lawson transformation to the system \eqref{DNLS-SVM-Formulation}, and obtain the following ESVM reformulation of \eqref{DNLS-equation}
\begin{equation}\label{DNLS-MSVM-Formulation}
\left\lbrace
\begin{aligned}
&\partial_t \Psi =  \exp(-\mathcal Lt)\mathcal N[\psi]+\sum_{j=1}^3\exp(-\mathcal Lt)\beta_j g_j[\psi],\\
&\frac{{\rm d}}{{\rm d}t}\mathcal E[\psi]=0,\quad\frac{{\rm d}}{{\rm d}t}\mathcal M[\psi]=0,\quad  \frac{{\rm d}}{{\rm d}t}\mathcal P[\psi]=0, \quad \psi=\exp(t\mathcal L)\Psi.
\end{aligned}\right.
  \end{equation}
Compared with the original system \eqref{DNLS-equi}, the ESVM system \eqref{DNLS-MSVM-Formulation} is an extended system since the additional supplementary variables $\beta_1$, $\beta_2$ and $\beta_3$ are introduced, and the consistence of \eqref{DNLS-MSVM-Formulation} and \eqref{DNLS-equi} will be discussed in the subsequent Section \ref{DNLS-STA}.

Subsequently, we employ the high-order prediction-correction Lawson Runge-Kutta method \eqref{eq:prediction-scheme} and \eqref{eq:correction-scheme} to discretize the system \eqref{DNLS-MSVM-Formulation} in time, and an $s$-stage  mass-, energy- and momentum-preserving mSP exponential integrator is presented.

\begin{itemize}
\item
   \textbf{High-order prediction:} Let ${\psi}^{(0)}_{ni}=\exp(c_i\tau\mathcal L) {\psi}^{n}$ and $M\in\mathbb{Z}^{+}$, we compute ${\psi}^{(\ell)}_{ni}$ iteratively as
\begin{align}\label{eq:ESVM-prediction-DNLS}
      {\psi}^{(\ell+1)}_{ni}=\exp(c_i\tau\mathcal L) {\psi}^{n}+\tau\sum_{j=1}^{s}a_{ij}\exp((c_i-c_j)\tau\mathcal L)\mathcal N[{\psi}_{nj}^{(\ell)}],\  \ell=0,1,\cdots,M-1.
\end{align}
\item\textbf{High-order correction:} update ${\psi}^{n+1}$ by
\begin{equation}\label{eq:ESVM-correction-DNLS}
\left\{
\begin{aligned}
&{k}_i=\mathcal N[{\psi}_{ni}^{(M)}]+\beta_1^ng_1[{\psi}^{(M)}_{ni}]+\beta_2^ng_2[{\psi}^{(M)}_{ni}]+\beta_3^ng_3[{\psi}^{(M)}_{ni}],\\
&{\psi}^{n+1}=\exp(\mathcal L\tau) {\psi}^{n}+\tau\sum_{i=1}^{s}b_i \exp((1-c_i)\tau\mathcal L) {k}_i,
\end{aligned}
\right.
\end{equation}
with suitable parameters $\beta_1^n, \beta_2^n$ and $\beta_3^n$ such that the energy, mass and momentum conservations
\begin{align}\label{eq:ESVM-conservations-DNLS}
\mathcal E[{\psi}^{n+1}]=\mathcal E[{\psi}^{n}],\quad \mathcal M[{\psi}^{n+1}]=\mathcal M[{\psi}^{n}],\quad \mathcal P[{\psi}^{n+1}]=\mathcal P[{\psi}^{n}]
\end{align}
hold true at $t_{n+1}$.
\end{itemize}

Finally, we show how to efficiently implement the proposed method \eqref{eq:ESVM-prediction-DNLS}-\eqref{eq:ESVM-correction-DNLS}. Let $\alpha_1=\tau\beta_1^n,\ \alpha_2=\tau\beta_2^n,\ \alpha_3=\tau\beta_3^n$, we can deduce from \eqref{eq:ESVM-prediction-DNLS} and \eqref{eq:ESVM-correction-DNLS}
\begin{equation}\label{eq:ESVM-Implement}
{\psi}^{n+1}=\widehat{\psi}^{n+1}+\alpha_1\Gamma_1^n+\alpha_2\Gamma_2^n+\alpha_3\Gamma_3^n,
\end{equation}
where
\begin{align}\label{eq:msp1}
&\widehat{\psi}^{n+1}=\exp(\tau\mathcal L){\psi}^{n}+\tau\sum_{i=1}^{s}b_i \exp((1-c_i)\tau\mathcal L)\mathcal N[{\psi}_{ni}^{(M)}],\\\label{eq:msp2}
& \Gamma_j^n=\sum_{i=1}^{s}b_i\exp((1-c_i)\tau\mathcal L)g_j[{\psi}^{(M)}_{ni}],\qquad j=1,2,3.
\end{align}
Inserting \eqref{eq:ESVM-Implement} into \eqref{eq:ESVM-conservations-DNLS}, we have
\begin{align}\label {eq:algebraic equation1}
&\mathcal G_1(\tau,\bma):=\mathcal E[\widehat{\psi}^{n+1}+\alpha_1\Gamma_1^n+\alpha_2\Gamma_2^n+\alpha_3\Gamma_3^n] -
\mathcal E[{\psi}^n]=0,\hfill\hfill\\
\label {eq:algebraic equation2}
&\mathcal G_2(\tau,\bma):=\mathcal M[\widehat{\psi}^{n+1}+\alpha_1\Gamma_1^n+\alpha_2\Gamma_2^n+\alpha_3\Gamma_3^n]-
\mathcal M[{\psi}^n]=0,\hfill\hfill\\
\label {eq:algebraic equation3}
&\mathcal G_3(\tau,\bma):=\mathcal P[\widehat{\psi}^{n+1}+\alpha_1\Gamma_1^n+\alpha_2\Gamma_2^n+\alpha_3\Gamma_3^n]-
\mathcal P[{\psi}^n]=0,
\end{align}
where $\bma=[\alpha_1,\alpha_2, \alpha_3]^{\top}$. Let
\begin{equation}\label{DNLS-G-alpha}
{\bf G}(\tau,\bma)=[\mathcal G_1(\tau,\bma),\mathcal G_2(\tau,\bma),\mathcal G_3(\tau,\bma)]^{\top},
 \end{equation}and the above algebraic equations can be solved by using Newton iteration method taking exact solution $\bma = {\bf 0}$ as the initial value, i.e.,
\begin{align}\label{sNLs-NI}
   \bma^{(\ell+1)}=\bma^{(\ell)}-[(J_\bma{\bf G})(\tau,\bma^{(\ell)})]^{-1}{\bf G}(\tau,\bma^{(\ell)}),\quad \bma^{(0)}= {\bf 0},\quad \ell=0,1,\cdots,
\end{align}\label{eq:mSP-NI}
 where the Jacobian of ${\bf G}(\tau,\bma)$ is explicitly given by
\begin{align}\label{DNLS-Jcobi-alpha}
(J_\bma{\bf G})(\tau,\bma):=2\left[\begin{array}{ccc}
 \Re\left(\frac{\delta \mathcal E}{\delta \psi^\ast},\Gamma_1^n\right)& \Re\left(\frac{\delta \mathcal E}{\delta \psi^\ast},\Gamma_2^n\right)& \Re\left(\frac{\delta \mathcal E}{\delta \psi^\ast},\Gamma_3^n\right)\vspace{1mm}\\
\Re\left(\frac{\delta \mathcal M}{\delta \psi^\ast},\Gamma_1^n\right) &  \Re\left(\frac{\delta \mathcal M}{\delta \psi^\ast},\Gamma_2^n\right)& \Re\left(\frac{\delta \mathcal M}{\delta \psi^\ast},\Gamma_3^n\right)\vspace{1mm}\\
\Re\left(\frac{\delta \mathcal P}{\delta \psi^\ast},\Gamma_1^n\right) &  \Re\left(\frac{\delta \mathcal P}{\delta \psi^\ast},\Gamma_2^n\right)& \Re\left(\frac{\delta \mathcal P}{\delta \psi^\ast},\Gamma_3^n\right)
\end{array} \right].
\end{align}

To sum up, from time $t_n$ to $t_{n+1}$,  the above semi-discrete \textbf{mSP}  algorithm is solved in {\bf Algorithm 1}.

\begin{table}[h]
		\begin{tabular}{l}\hline
			\\[-2.3ex]{{\bf Algorithm 1}: semi-discrete \textbf{mSP} algorithm}\\[0.5ex]
\hline\\[-8.5pt]
[{\bf High-order prediction}] \textbf{Input}: Given $a_{ij},\ b_i,\ 1\le i,j\le s$, $\psi^{n}$, $M \in \mathbb{Z}^+$, and ${\psi}^{(0)}_{ni}=\exp(c_i\tau\mathcal L) {\psi}^{n}$.\\[0.8ex]
{For $\ell = 0, \dots, M-1$, computing $\psi_{ni}^{(\ell+1)}$  by employing \eqref{eq:ESVM-prediction-DNLS}.}\\[0.8ex]
\textbf{Output}: $\psi_{ni}^{(M)},\quad i=1,2,\cdots,s$.\\[1ex]
\hline\\[-8.5pt]
[{\bf High-order correction}]  \textbf{Input}: Given $\psi^{n}$, $a_{ij},\ b_i,\ 1\le i,j\le s$ and $\psi_{ni}^{(M)}$.\\[0.8ex]

{1: Computing  $\widehat{\psi}^{n+1}$ and $\Gamma_j^n,\ j=1,2,3$ by \eqref{eq:msp1}-\eqref{eq:msp2}.}\\[0.8ex]

{2: ${\bm \alpha}$ is obtained by employing the Newton iteration method \eqref{sNLs-NI}.}\\[0.8ex]

{3: $\psi^{n+1}$ is obtained by computing \eqref{eq:ESVM-Implement}.} \\[1ex]
\textbf{Output}: $\psi^{n+1}$.\\[1ex]
\hline
		\end{tabular}
	\end{table}

\subsection{Spatial discretization}\label{Sec:ESVM-DNLS:3-2}
In this subsection, a fully discrete mass-, energy- and momentum-preserving exponential integrator is presented.
Due to the mass, energy and momentum being given explicitly in the ESVM system, it is more convenient and flexible to choose a suitable spatial discretization for the problem with the given boundary condition. Here we consider the periodic boundary condition, thus the Fourier spectral method \cite{CQ-ETNA-2001,ST06} is a good choice to approximate the wave function thanks to the spectral accuracy and the discrete Fast Fourier transform (FFT).

To begin with, we let $\Omega=[a,b]$, and choose the spatial step $h=(b-a)/N$ with an even positive integer $N$; denote $\psi_{j}^n$ and $(\psi_{j})_{ni}$ as the numerical approximations of  $\psi(x_j,t_n)$ and $\psi(x_j,t_{ni})$for $j=0,1,\cdots,N,\ n=0,1,\cdots,$ respectively. Let ${\bm \psi}^n=[\psi_{0}^n\ \psi_{1}^n\ \cdots\ \psi_{N-1}^n]^\top$
be the grid vector function, and for any two grid vector functions ${\bm \psi}^n$ and ${\bm \phi}^n$, we define the discrete inner product and $l^2$-norm as, respectively,
\begin{equation*}
({\bm \psi}^n,{\bm \phi}^n)_{h}=h\sum_{j=0}^{N-1} \psi_{j}^n(\phi_{j}^n)^\ast,\qquad
\|{\bm \psi}^n\|_{h}^2=({\bm \psi}^n,{\bm \psi}^n)_{h}.
\end{equation*}
Additionally, we also introduce another operator $``\odot"$ for element by element multiplication between two vector functions of same sizes as
\begin{equation*}
({\bm \psi}^n\odot{\bm \phi}^n)_{j}=({\bm \phi}^n\odot{\bm \psi}^n)_{j}=(\psi_{j}^n\cdot\phi_{j}^n),\qquad |{\bm \psi}|^2= {\bm \psi}^n\odot({\bm \psi}^n)^\ast.
\end{equation*}

Then, we denote the interpolation space as
\begin{align*}
&\mathcal V_N=\text{span}\{g_{j},\quad 0\leq j\leq N-1\},
\end{align*}where $g_{j}(x)$ is a trigonometric polynomial of degree $N/2$, given explicitly by
\begin{align*}
  &g_{j}(x)=\frac{1}{N}\sum_{l=-N/2}^{N/2}\frac{1}{c_{l}}e^{\text{i}l\mu (x-x_{j})},\qquad c_{l}=\left \{
 \aligned
 &1,\ |l|<\frac{N}{2},\\
 &2,\ |l|=\frac{N}{2},
 \endaligned
 \right.\qquad \mu=\frac{2\pi}{b-a}.
\end{align*}
According to \cite{ST06}, the interpolation operator is defined as $\mathcal{I}_{N}: C(\Omega)\to \mathcal V_N$:
\begin{align}\label{SP-interpolation}
\mathcal{I}_{N}\psi(x,t)=\sum_{j=0}^{N-1}\psi(x_j,t)g_{j}(x).
\end{align} Importantly, the terms $\partial_x\psi,\ \partial_{xx}\psi$ and $\exp(t\mathcal L)\psi$ ($\mathcal L={\rm i}\partial_{xx}$) are
approximated as follows
\begin{align*}
&\partial_x\psi(x_{j},t)\approx \frac{\partial \mathcal{I}_{N}\psi(x_{j},t)}{\partial x}
=({\bm D}_1{\bm \psi}(t))_j,\quad \partial_{xx}\psi(x_{j},t)\approx\frac{\partial^{2} \mathcal{I}_{N}\psi(x_{j},t)}{\partial x^2}
=({\bm D}_2{\bm \psi}(t))_j,\\
&\exp(t\mathcal L)\psi(x_{j},t)\approx \exp(t\mathcal L)\mathcal{I}_{N}\psi(x_{j},t)=(\exp(t\mathcal L_h){\bm \psi}(t))_j,\quad \mathcal L_h = {\rm i}{\bm D}_2,
\end{align*}
where ${\bm D}_1$ and ${\bm D}_2$ are the first-order and second-order spectral differentiation matrices, respectively, given by \cite{CQ-ETNA-2001}
\begin{align*}
({\bm D}_{1})_{j,k}=\left \{
 \aligned
 &\frac{1}{2}\mu (-1)^{j+k}\cot(\mu \frac{x_{j}-x_{k}}{2}),\ &j\neq k,\\
 &0,\quad \quad \quad \quad ~ &j=k,
 \endaligned
 \right.\
 ({\bm D}_{2})_{j,k}=
 \left \{
 \aligned
 &\frac{1}{2}\mu^{2} (-1)^{j+k+1}\csc^{2}(\mu \frac{x_{j}-x_{k}}{2}),\ &j\neq k,\\
 &-\mu^{2}\frac{N^{2}+2}{12},\quad \quad \quad \quad ~ &j=k.
 \endaligned
 \right.
 \end{align*}

\begin{rmk} We apply the standard Fourier pseudo-spectral method to approximate the mass-, energy and momentum \eqref{DNLS-mass}-\eqref{DNLS-momentum} in space, and then obtain the semi-discrete mass, energy and momentum as follows:
\begin{align*}
&\mathcal{M}_h[{\bm \psi}]=\|{\bm\psi}\|_h^2,\\
&\mathcal{E}_h[{\bm\psi}] = (-{\bm D}_2{\bm \psi},{\bm \psi})_h-\frac{3}{2}\Im(|{\bm\psi}|^2\odot{\bm\psi},{\bm D_1}{\bm \psi})_h+\frac{1}{2}(|{\bm\psi}|^6,{\bf 1})_h,\\
&\mathcal{P}_h[{\bm\psi}]=\Im({\bm D_1}{\bm\psi},{\bm \psi})_h+\frac{1}{2}(|{\bm \psi}|^4,{\bf 1})_h.
\end{align*}
\end{rmk}

\begin{rmk}\label{SP-rmk-2.2} According to \cite{ST06}, we have
   \begin{align*}
&{\bm D}_1=\mathcal{F}^{-1}\Lambda_{{\bm D}_1} \mathcal{F},\qquad \Lambda_{{\bm D}_1}=\text{\rm i}\mu\text{\rm diag}\Big[0,1,\cdots,\frac{N}{2}-1,0,-\frac{N}{2}+1,\cdots,-2,-1\Big],\\
& {\bm D}_2=\mathcal{F}^{-1}\Lambda_{{\bm D}_2} \mathcal{F},\qquad \Lambda_{{\bm D}_2}
=\big(\text{\rm i}\mu\big)^2\text{\rm diag}\Big[0,1,\cdots,\frac{N}{2},-\frac{N}{2}+1,\cdots,-2,-1\Big]^2,
\end{align*}
where $\mathcal{F}$ is the discrete Fourier transform (DFT) and $\mathcal{F}^{-1}$ represents the inverse discrete Fourier transform. Additionally, based on the above equality, we can deduce
\begin{equation*}
\exp(t\mathcal L_h) = \mathcal{F}^{-1}\exp({\rm i}\Lambda_{{\bm D}_2}t)\mathcal{F},
\end{equation*}
which implies the exponential matrix ($\exp(t\mathcal L_h)$) can be efficiently computed by using the discrete Fourier transform.
   \end{rmk}

 Subsequently, the standard Fourier pseudo-spectral method is employed to the proposed method \eqref{eq:ESVM-prediction-DNLS}-\eqref{eq:ESVM-correction-DNLS} for spatial discretizations, we obtain the following fully discrete scheme.
\begin{itemize}
\item
   \textbf{High-order prediction:} Let ${\bm\psi}^{(0)}_{ni}=\exp(c_i\tau\mathcal L_h) {\bm\psi}^{n},\ i=1,2,\cdots,s$ and $M\in\mathbb{Z}^{+}$, we compute ${\bm\psi}^{(\ell)}_{ni}$ iteratively as
\begin{align}\label{eq:msp-prediction-DNLS}
      &{\bm\psi}^{(\ell+1)}_{ni}=\exp(c_i\tau\mathcal L_h) {\bm\psi}^{n}+\tau\sum_{j=1}^{s}a_{ij}\exp((c_i-c_j)\tau\mathcal L_h) \mathcal N[{\bm \psi}_{ni}^{(\ell)}],\quad \ell=0,1,\cdots,M-1,
\end{align}
where $\mathcal N[{\bm \psi}_{ni}]=-{\bm D}_1\big(|{\bm \psi}_{ni}|^2\odot{\bm \psi}_{ni}\big)$.
\item\textbf{High-order correction:} update ${\bm \psi}^{n+1}$ by
\begin{equation}\label{eq:msp-correction-DNLS}
\left\{
\begin{aligned}
&{\bm k}_i=\mathcal N[{\bm \psi}_{ni}^{(M)}]+\beta_1^ng_1[{\bm\psi}^{(M)}_{ni}]+\beta_2^ng_2[{\bm\psi}^{(M)}_{ni}]+\beta_3^ng_3[{\bm\psi}^{(M)}_{ni}],\\
&{\bm\psi}^{n+1}=\exp(\tau \mathcal L_h) {\bm\psi}^{n}+\tau\sum_{i=1}^{s}b_i \exp((1-c_i)\tau\mathcal L_h) {\bm k} _i,
\end{aligned}
\right.
\end{equation}
with suitable parameters $\beta_1^n, \beta_2^n$ and $\beta_3^n$ such that the energy, mass and momentum conservations
\begin{align}\label{eq:ffESVM-conservations-DNLS}
\mathcal E_h[{\psi}^{n+1}]=\mathcal E_h[{\psi}^{n}],\quad \mathcal M_h[{\psi}^{n+1}]=\mathcal M_h[{\psi}^{n}],\quad \mathcal P_h[{\psi}^{n+1}]=\mathcal P_h[{\psi}^{n}]
\end{align}
hold true at $t_{n+1}$.
\end{itemize}

Similar to the proposed semi-discrete method \eqref{eq:ESVM-prediction-DNLS}-\eqref{eq:ESVM-correction-DNLS}, the proposed fully-discrete mSP scheme \eqref{eq:msp-prediction-DNLS}-\eqref{eq:msp-correction-DNLS} can be solved efficiently. Specifically, it follows from \eqref{eq:msp-prediction-DNLS}-\eqref{eq:msp-correction-DNLS} that
\begin{equation}\label{eq:fdNLS-Implement}
{\bm \psi}^{n+1}=\widehat{\bm\psi}^{n+1}+\alpha_1{\bm\Gamma}_1^n+\alpha_2{\bm\Gamma}_2^n+\alpha_3{\bm\Gamma}_3^n,
\end{equation}
where
\begin{align}\label{DNLS:feq:msp1}
&\widehat{\bm\psi}^{n+1}=\exp(\tau\mathcal L_h){\bm\psi}^{n}+\tau\sum_{i=1}^{s}b_i \exp((1-c_i)\tau\mathcal L_h)\left(-{\bf D}_1\big(|{\bm\psi}_{ni}^{(M)}|^2\odot{\bm\psi}_{ni}^{(M)}\big)\right),\\\label{DNLS:feq:msp2}
& {\bm\Gamma}_j^n=\sum_{i=1}^{s}b_i\exp((1-c_i)\tau\mathcal L_h)g_j[{\bm\psi}^{(M)}_{ni}],\qquad j=1,2,3,
\end{align}
Inserting \eqref{eq:fdNLS-Implement} into \eqref{eq:ffESVM-conservations-DNLS}, we obtain
\begin{align}\label{eq:DNLS-algebraic1}
&\mathcal G_{1,h}(\tau,\bma):=\mathcal E_h[\widehat{\bm\psi}^{n+1}+\alpha_1{\bm\Gamma}_1^n+\alpha_2{\bm\Gamma}_2^n+\alpha_3{\bm\Gamma}_3^n] -
\mathcal E_h[{\bm\psi}^n]=0,\hfill\hfill\\\label{eq:DNLS-algebraic2}
&\mathcal G_{2,h}(\tau,\bma):=\mathcal M_h[\widehat{\bm\psi}^{n+1}+\alpha_1{\bm\Gamma}_1^n+\alpha_2{\bm\Gamma}_2^n+\alpha_3{\bm\Gamma}_3^n]-
\mathcal M_h[{\bm\psi}^n]=0,\hfill\hfill\\\label{eq:DNLS-algebraic3}
&\mathcal G_{3,h}(\tau,\bma):=\mathcal P_h[\widehat{\bm\psi}^{n+1}+\alpha_1{\bm\Gamma}_1^n+\alpha_2{\bm\Gamma}_2^n+\alpha_3{\bm\Gamma}_3^n]-
\mathcal P_h[{\bm\psi}^n]=0.
\end{align}
The above algebraic equations is solved by employing the following Newton iteration method, as follows:
\begin{align*}
   \bma^{(\ell+1)}=\bma^{(\ell)}-[(J_{\bma,h}{\bf G}_h)(\tau,\bma^{(\ell)})]^{-1}{\bf G}_h(\tau,\bma^{(\ell)}),\quad \bma^{(0)}= {\bf 0},\quad \ell=0,1,\cdots,
\end{align*}
where ${\bf G}_h(\tau,\bma)=[\mathcal G_{1,h}(\tau,\bma),\mathcal G_{2,h}(\tau,\bma),\mathcal G_{3,h}(\tau,\bma)]^{\top}$  and \begin{align}\label{fully-dis-J-ah}
(J_{\bma,h}{\bf G}_h)(\tau,\bma):=2\left[\begin{array}{ccc}
 \Re\left(\nabla_{{\bm\psi}^\ast}\mathcal E_h,{\bm\Gamma}_{1}^n\right)_h& \Re\left(\nabla_{{\bm\psi}^\ast}\mathcal E_h,{\bm\Gamma}_{2}^n\right)_h& \Re\left(\nabla_{{\bm\psi}^\ast}\mathcal E_h,{\bm\Gamma}_{3}^n\right)_h\vspace{1mm}\\
\Re\left(\nabla_{{\bm\psi}^\ast}\mathcal M_h,{\bm\Gamma}_{1}^n\right)_h& \Re\left(\nabla_{{\bm\psi}^\ast}\mathcal M_h,{\bm\Gamma}_{2}^n\right)_h& \Re\left(\nabla_{{\bm\psi}^\ast}\mathcal M_h,{\bm\Gamma}_{3}^n\right)_h\vspace{1mm}\\
\Re\left(\nabla_{{\bm\psi}^\ast}\mathcal P_h,{\bm\Gamma}_{1}^n\right)_h& \Re\left(\nabla_{{\bm\psi}^\ast}\mathcal P_h,{\bm\Gamma}_{2}^n\right)_h& \Re\left(\nabla_{{\bm\psi}^\ast}\mathcal P_h,{\bm\Gamma}_{3}^n\right)_h
\end{array} \right].
\end{align}

From time $t_n$ to $t_{n+1}$, the proposed fully-discrete  mSP algorithm \eqref{eq:msp-prediction-DNLS}-\eqref{eq:msp-correction-DNLS} is solved in {\bf Algorithm 2}.
\begin{table*}[h]
\begin{tabular}{l}
\\\hline{\bf Algorithm 2}: Fully-discrete \textbf{mSP} algorithm
\\\hline
[{\bf High-order prediction}] \textbf{Input}: Given $a_{ij},\ b_i,\ 1\le i,j\le s$, ${\bm\psi}^{n}$, $M \in \mathbb{Z}^+$, and ${\bm\psi}^{(0)}_{ni}=\exp(c_i\tau\mathcal L_h) {\bm\psi}^{n}$.\\[0.8ex]
{For $\ell = 0, \dots, M-1$, computing ${\bm \psi}_{ni}^{(\ell+1)}$ by using \eqref{eq:msp-prediction-DNLS}.}\\[0.8ex]
\textbf{Output}: ${\bm\psi}_{ni}^{(M)},\qquad i=1,2,\cdots,s$.\\[1ex]
\hline\\[-8.5pt]
[{\bf High-order correction}]  \textbf{Input}: Given $a_{ij},\ b_i,\ 1\le i,j\le s$, ${\bm\psi}^{n}$ and ${\bm\psi}_{ni}^{(M)}$.\\[0.8ex]

1: Computing  $\widehat{\bm\psi}^{n+1}$ and ${\bm\Gamma}_j^n,\ j=1,2,3,$ by using \eqref{DNLS:feq:msp1}-\eqref{DNLS:feq:msp2}. \\

3: ${\bm\psi}^{n+1}$ is obtained by \eqref{eq:fdNLS-Implement}.  \\[1ex]
\textbf{Output}: ${\bm\psi}^{n+1}$.\\[1ex]
\hline
\end{tabular}
\end{table*}

\section{Some theoretical analysis}\label{DNLS-STA}
 In this section, we first investigate the consistence of \eqref{DNLS-MSVM-Formulation} and  \eqref{DNLS-equation}, and then show the existence and uniqueness of ${\bma}$ in \eqref{eq:algebraic equation1}-\eqref{eq:algebraic equation3}. Finally, we analyze the convergence order in local truncation error of \eqref{eq:ESVM-prediction-DNLS}-\eqref{eq:ESVM-correction-DNLS} and estimate the magnitudes of supplementary variables $\beta_1^n$, $\beta_2^n$ and $\beta_3^n$, respectively.

\subsection{Consistence}

As mentioned above, the ESVM system \eqref{DNLS-MSVM-Formulation} is an extended system, thus the consistence of \eqref{DNLS-MSVM-Formulation} and \eqref{DNLS-equi} plays a key role in the  construction of the mass-, energy- and momentum-preserving schemes for the system \eqref{DNLS-equi}. Importantly, we note that the ESVM reformulation \eqref{DNLS-MSVM-Formulation} is equivalent to the original system \eqref{DNLS-equation} as ${\bm\beta}=[\beta_1\ \beta_2\ \beta_3]^{\top}={\bm 0}$.
\begin{thm}\label{DNLS-thm-3.1}
Let \begin{equation*}\mathcal J[\psi] = 2\begin{bmatrix}
	\|\frac{\delta \mathcal{E}}{\delta {\psi}^\ast}\|^2 &
	\Re\left(\frac{\delta \mathcal{M}}{\delta {\psi}^\ast}, \frac{\delta \mathcal{E}}{\delta {\psi}^\ast} \right) &
	\Re\left(\frac{\delta \mathcal{P}}{\delta {\psi}^\ast}, \frac{\delta \mathcal{E}}{\delta {\psi}^\ast} \right)\\[7pt]
	\Re\left(\frac{\delta \mathcal{E}}{\delta {\psi}^\ast} , \frac{\delta \mathcal{M}}{\delta {\psi}^\ast}  \right)
	&\|\frac{\delta \mathcal{M}}{\delta {\psi}^\ast}\|^2& \Re\left(\frac{\delta \mathcal{P}}{\delta {\psi}^\ast}, \frac{\delta \mathcal{M}}{\delta {\psi}^\ast} \right)\\[7pt]
	\Re\left(\frac{\delta \mathcal{E}}{\delta {\psi}^\ast}, \frac{\delta \mathcal{P}}{\delta {\psi}^\ast}  \right)
	&\Re\left(\frac{\delta \mathcal{M}}{\delta {\psi}^\ast}, \frac{\delta \mathcal{P}}{\delta {\psi}^\ast} \right)& \|\frac{\delta \mathcal{P}}{\delta {\psi}^\ast}\|^2
\end{bmatrix},
\end{equation*}
then the ESVM reformulation \eqref{DNLS-MSVM-Formulation} is equivalent to the original system \eqref{DNLS-equation} if and only if $ \frac{\delta \mathcal E}{\delta {\psi}^\ast}$, $ \frac{\delta \mathcal M}{\delta {\psi}^\ast}$ and $ \frac{\delta \mathcal P}{\delta {\psi}^\ast}$ are linearly independent.
\end{thm}

\begin{prf}
	With the use of Eqns \eqref{DNLS-conservation1}-\eqref{DNLS-conservation2} and the second equality of \eqref{DNLS-MSVM-Formulation}, we have
\begin{align*}
	0=\frac{\mathrm{d}}{\mathrm{d}t}\mathcal{E} &=
	2\Re\left( \partial_t \psi, \frac{\delta \mathcal{E}}{\delta {\psi}^\ast}    \right)  \\
	&=2\Re\left(\mathcal L\psi+\mathcal N[\psi]+\beta_1 g_1[\psi]+\beta_2g_2[\psi]+\beta_3 g_3[\psi], \frac{\delta \mathcal{E}}{\delta {\psi}^\ast}    \right)  \\
	&=2\beta_1\Re\left(g_1[\psi], \frac{\delta \mathcal{E}}{\delta {\psi}^\ast}    \right)+2\beta_2\Re\left(g_2[\psi], \frac{\delta \mathcal{E}}{\delta {\psi}^\ast}    \right)+2\beta_3 \Re\left(g_3[\psi], \frac{\delta \mathcal{E}}{\delta {\psi}^\ast}    \right).
\end{align*}
Analogously, we can get
\begin{align*}
0=\frac{\mathrm{d}}{\mathrm{d}t}\mathcal{M} &=
2\Re\left( \partial_t \psi, \frac{\delta \mathcal{M}}{\delta {\psi}^\ast}    \right)  \\\\
&=2\beta_1 \Re\left(g_1[\psi], \frac{\delta \mathcal{M}}{\delta {\psi}^\ast}    \right)+2\beta_2\Re\left(g_2[\psi], \frac{\delta \mathcal{M}}{\delta {\psi}^\ast}    \right)+2\beta_3 \Re\left(g_3[\psi], \frac{\delta \mathcal{M}}{\delta {\psi}^\ast}    \right).
\end{align*}
and \begin{align*}
0=\frac{\mathrm{d}}{\mathrm{d}t}\mathcal{P} &=
2\Re\left( \partial_t \psi, \frac{\delta \mathcal{P}}{\delta {\psi}^\ast}    \right)  \\\\
&=2\beta_1 \Re\left(g_1[\psi], \frac{\delta \mathcal{P}}{\delta {\psi}^\ast}    \right)+2\beta_2\Re\left(g_2[\psi], \frac{\delta \mathcal{P}}{\delta {\psi}^\ast}    \right)+2\beta_3 \Re\left(g_3[\psi], \frac{\delta \mathcal{P}}{\delta {\psi}^\ast}    \right).
\end{align*}
It follows from the above equations that
\begin{equation*}
J[\psi]{\bm \beta} = {\bf 0},
\end{equation*}
which implies that the above equations admit a unique zero solution if and only if $\mathcal J[\psi]$ is nonsingular. Note that $\mathcal J[\psi]$ is nonsingular if and only if  $ \frac{\delta \mathcal E}{\delta {\psi}^\ast}$, $ \frac{\delta \mathcal M}{\delta {\psi}^\ast}$ and $ \frac{\delta \mathcal P}{\delta {\psi}^\ast}$ are linearly independent. The proof is completed. \qed
\end{prf}

\subsection{Existence, uniqueness and local truncation error}

In this subsection, following by the ideas presented in \cite{CHMR06,GJZ2024,HWGCMAME2023}, we first show the existence and uniqueness of ${\bma}$ in \eqref{eq:algebraic equation1}-\eqref{eq:algebraic equation3}. Then, the convergence order in local truncation error of the semi-discrete mSP method \eqref{eq:ESVM-prediction-DNLS} and \eqref{eq:ESVM-correction-DNLS} will be explored. Finally, the estimates on the magnitudes of supplementary variable ${\bma}$ are carried out.
	\begin{thm}[Local existence and uniqueness of ${\bma}$]\label{ESVM:thm:dNLS} Suppose the Jacobian $(J_\bma {\bf G})(0,{\bm 0})$ (see \eqref{DNLS-Jcobi-alpha}) is nonsingular, there exists a $\tau^*>0$ such that the nonlinear algebraic equations \eqref{eq:algebraic equation1}-\eqref{eq:algebraic equation3} admit a unique function $\bma=\bma(\tau)$ for all $\tau\in[0,\tau^*]$.
	\end{thm}
	\begin{prf}With noting the smoothness of ${\bf G}(\tau,\bma)$ (see \eqref{DNLS-G-alpha}) and condition
		\begin{equation*}
			{\bf G}(0,{\bm 0})={\bm 0},\quad \Big|(J_\bma{\bf G})(0,{\bm 0})\Big|\not=0,
		\end{equation*}
		together with the implicit function theorem, there exists a $\tau^*>0$ such
		that the nonlinear algebraic system ${\bf G}(\tau,\bma)={\bm 0}$ admits a unique smooth function
		$\bma=\bma(\tau)$ satisfying $\bma({\bm 0})={\bm 0}$ and ${\bf G}(\tau,\bma(\tau))={\bm 0}$ for all $\tau\in[0,\tau^*]$.
	\end{prf}

	\begin{rmk}
		It is clear to see that $$(J_\bma{\bf G})(0,{\bm 0})= 2\begin{bmatrix}
	\|\frac{\delta \mathcal{E}}{\delta {\psi}^\ast}[\psi^n]\|^2 &
	\Re\left(\frac{\delta \mathcal{M}}{\delta {\psi}^\ast}[\psi^n], \frac{\delta \mathcal{E}}{\delta {\psi}^\ast}[\psi^n] \right) &
	\Re\left(\frac{\delta \mathcal{P}}{\delta {\psi}^\ast}[\psi^n], \frac{\delta \mathcal{E}}{\delta {\psi}^\ast}[\psi^n] \right)\\[7pt]
	\Re\left(\frac{\delta \mathcal{E}}{\delta {\psi}^\ast}[\psi^n] , \frac{\delta \mathcal{M}}{\delta {\psi}^\ast}[\psi^n]  \right)
	&\|\frac{\delta \mathcal{M}}{\delta {\psi}^\ast}[\psi^n]\|^2& \Re\left(\frac{\delta \mathcal{P}}{\delta {\psi}^\ast}[\psi^n], \frac{\delta \mathcal{M}}{\delta {\psi}^\ast}[\psi^n] \right)\\[7pt]
	\Re\left(\frac{\delta \mathcal{E}}{\delta {\psi}^\ast}[\psi^n] , \frac{\delta \mathcal{P}}{\delta {\psi}^\ast}[\psi^n]  \right)
	&\Re\left(\frac{\delta \mathcal{M}}{\delta {\psi}^\ast}[\psi^n], \frac{\delta \mathcal{P}}{\delta {\psi}^\ast}[\psi^n] \right)& \|\frac{\delta \mathcal{P}}{\delta {\psi}^\ast}[\psi^n]\|^2
\end{bmatrix} = \mathcal J[\psi^n],$$
		which implies that the condition for the existence and uniqueness of solution in semi-discrete scenario is consistent with that in continuous context (see Theorem \ref{DNLS-thm-3.1}).
	\end{rmk}

\begin{thm} Assume the solution $\psi(x,t)$ of the DNLS equation \eqref{DNLS-equation} is smooth enough and the underlying RK method is of order $p$ and $\Big|(J_\bma{\bf G})(0,{\bm 0})\Big|\not=0$, the semi-discrete mSP method \eqref{eq:ESVM-prediction-DNLS} and \eqref{eq:ESVM-correction-DNLS} can achieve order $\hat p=\min\{p,M+1\}$ in time.
\end{thm}

\begin{prf}
		We note that $\widehat{\psi}^{n+1}$ satisfies
		\begin{equation}\label{eq:ESVM-hat-psi-n}
			\widehat{\psi}^{n+1}=\exp(\tau\mathcal L){\psi}^{n}+\tau\sum_{i=1}^{s}b_i \exp((1-c_i)\tau\mathcal L)\mathcal N[{\psi}_{ni}^{(M)}]
		\end{equation}
		With the aid of Theorems 2 and 3 in \cite{GJZ2024}, we have
		\begin{equation}\label{eq:ESVM-hat-psi}
			\widehat{\psi}^{n+1}=\psi(t_{n+1})+\mathcal{O}(\tau^{\min\{p,M+1\}+1}).
		\end{equation}
		we then employ the Taylor expansion to obtain
		\begin{align*}
			\mathcal E[\widehat{\psi}^{n+1}]=\mathcal E[\psi(t_{n+1})]+\mathcal{O}(\tau^{\min\{p,M+1\}+1}),\quad
            \mathcal M[\widehat{\psi}^{n+1}]=\mathcal M[\psi(t_{n+1})]+\mathcal{O}(\tau^{\min\{p,M+1\}+1}),
		\end{align*}
		and
		\begin{equation*}
			\mathcal P[\widehat{\psi}^{n+1}]=\mathcal P[\psi(t_{n+1})]+\mathcal{O}(\tau^{\min\{p,M+1\}+1}).
		\end{equation*}
		Since $\psi(t_{n+1})$ is the local solution of \eqref{DNLS-equi}, we have
		\begin{equation*}
			\psi(t_n)=\psi^n,\quad \mathcal E[\psi(t_{n+1})]=\mathcal E[\psi^n],\quad \mathcal M[\psi(t_{n+1})]=\mathcal M[\psi^n],\quad \mathcal P[\psi(t_{n+1})]=\mathcal P[\psi^n].
		\end{equation*}
		Thus, we can deduce
		\begin{equation*}
			\mathcal E[\widehat{\psi}^{n+1}]-\mathcal E[\psi^n]=\mathcal{O}(\tau^{\min\{p,M+1\}+1}),\ \mathcal M[\widehat{\psi}^{n+1}]-\mathcal M[\psi^n]=\mathcal{O}(\tau^{\min\{p,M+1\}+1}),
		\end{equation*}
and
\begin{equation*}
			\mathcal P[\widehat{\psi}^{n+1}]-\mathcal P[\psi^n]=\mathcal{O}(\tau^{\min\{p,M+1\}+1}).
		\end{equation*}
		On the other hand, we expand ${\bf G}(\tau,\bma)$ at $\bma={\textbf 0}$ as
		\begin{equation*}
			{\bf G}(\tau,\bma)={\bf G}(\tau,{\bm 0})+(J_\bma {\bf G})(\tau,{\bm 0})\bma+\mathcal{O}(|\bma|^2),
		\end{equation*}
		where
		\begin{equation*}
			{\bf G}(\tau,{\bm 0})=(\mathcal E[\widehat{\psi}^{n+1}]-\mathcal E[\psi^n],\mathcal M[\widehat{\psi}^{n+1}]-\mathcal M[\psi^n],\mathcal P[\widehat{\psi}^{n+1}]-\mathcal P[\psi^n])^{\top}=\mathcal{O}(\tau^{\min\{p,M+1\}+1}),
		\end{equation*}
and
		\begin{equation*}
			(J_\bma{\bf G})(\tau,{\bm 0})=(J_\bma {\bf G})(0,{\bm 0})+\mathcal{O}(\tau).
		\end{equation*}
		This implies
		\begin{equation}\label{eq:ESVM-alpha}
			\alpha_1=\alpha_1(\tau)=\mathcal{O}(\tau^{\hat{p}+1}),\ \alpha_2=\alpha_2(\tau)=\mathcal{O}(\tau^{\hat{p}+1}),\ \alpha_3=\alpha_3(\tau)=\mathcal{O}(\tau^{\hat{p}+1}),\ \hat{p}=\min\{p,M+1\}.
		\end{equation}
		Finally, we use \eqref{eq:ESVM-Implement} and \eqref{eq:ESVM-hat-psi}-\eqref{eq:ESVM-alpha} to obtain
		$\psi^{n+1}-\psi(t_{n+1})=\mathcal{O}(\tau^{\hat{p}+1})$.
	\end{prf}
	
	\begin{rmk}[Estimates on the magnitudes of supplementary variables $\beta_1$, $\beta_2$ and $\beta_3$]
		\label{esvm-rmk-alpha}
		According to \eqref{eq:ESVM-alpha}, we have the following estimates
		\begin{equation}\label{eq:ESVM-beta1-beta2-beta3}
			\beta_i^n = \mathcal{O}(\tau^{\hat{p}}),\quad i = 1,\ 2,\ 3,
		\end{equation}
		which means three variables approach zero with the same integer order $\hat{p}$.
		
\vspace{0.3cm}
		
		Actually, the equality \eqref{eq:ESVM-beta1-beta2-beta3} can been rewritten as
		$\beta_i^n = C_i(\psi^n) \tau^{\hat{p}} + \mathcal{O}(\tau^{\hat{p}+1})$.
		Note that $\psi^n = \psi(t_n) + \mathcal{O}(\tau)$ in the global error analysis, we then can deduce
		\begin{equation}\label{eq:ESVM-beta1-beta2-3}
			\beta_i^n = C_i(\psi(t_n)) \tau^{\hat{p}} + \mathcal{O}(\tau^{\hat{p}+1}),
		\end{equation}
		which implies immediately that the convergence orders are $\hat{p}$.

	\end{rmk}

\begin{rmk}Suppose that $g_1[\psi],\ g_2[\psi]$ and $g_3[\psi]$ are linearly independent and satisfy the Haar condition at the mesh grid $\{x_j\}_0^{N-1}$, then there exists a $\tau^*>0$ such that the nonlinear algebraic equations \eqref{eq:DNLS-algebraic1}-\eqref{eq:DNLS-algebraic3} admit a unique function $\bma=\bma(\tau)$ for all $\tau\in[0,\tau^*]$.

\vspace{0.3cm}
Notice that it follows from \eqref{fully-dis-J-ah}  \begin{align*}
&(J_{\bma,h}{\bf G}_h)(0,{\bm 0}):=\\
&~~2\left[\begin{array}{ccc}
 \|\nabla_{{\bm\psi}^\ast}\mathcal E_h[{\bm \psi}^n]\|_h^2& \Re\left(\nabla_{{\bm\psi}^\ast}\mathcal M_h[{\bm \psi}^n],\nabla_{{\bm\psi}^\ast}\mathcal E_h[{\bm \psi}^n]\right)_h& \Re\left(\nabla_{{\bm\psi}^\ast}\mathcal P_h[{\bm \psi}^n],\nabla_{{\bm\psi}^\ast}\mathcal E_h[{\bm \psi}^n]\right)_h\vspace{1mm}\\
\Re\left(\nabla_{{\bm\psi}^\ast}\mathcal E_h[{\bm \psi}^n],\nabla_{{\bm\psi}^\ast}\mathcal M_h[{\bm \psi}^n]\right)_h& \|\nabla_{{\bm\psi}^\ast}\mathcal M_h[{\bm \psi}^n]\|_h^2& \Re\left(\nabla_{{\bm\psi}^\ast}\mathcal P_h[{\bm \psi}^n],\nabla_{{\bm\psi}^\ast}\mathcal M_h[{\bm \psi}^n]\right)_h\vspace{1mm}\\
\Re\left(\nabla_{{\bm\psi}^\ast}\mathcal E_h[{\bm \psi}^n],\nabla_{{\bm\psi}^\ast}\mathcal P_h[{\bm \psi}^n]\right)_h& \Re\left(\nabla_{{\bm\psi}^\ast}\mathcal M_h[{\bm \psi}^n],\nabla_{{\bm\psi}^\ast}\mathcal P_h[{\bm \psi}^n]\right)_h& \|\nabla_{{\bm\psi}^\ast}\mathcal P_h[{\bm \psi}^n]\|_h^2
\end{array} \right],
\end{align*}
and $(J_{\bma,h}{\bf G}_h)(0,{\bm 0})$  is nonsingular if and only if $g_1[\psi],\ g_2[\psi]$ and $g_3[\psi]$ are linearly independent and satisfy the Haar condition at the mesh grid $\{x_j\}_0^{N-1}$. Then, by a similar argument as in Theorem \ref{ESVM:thm:dNLS}, we obtain the local existence and uniqueness of ${\bma}$ in \eqref{eq:DNLS-algebraic1}-\eqref{eq:DNLS-algebraic3}. However, an optimal error estimate for the proposed fully discrete scheme is somehow challenging and technically involve because of the multiple supplementary variables and the Lawson Runge-Kutta method. Importantly, some techniques presented in \cite{CSW-SIAMJNA-2025,YLFZ-CWA-2025} are highly referential.
	\end{rmk}

\section{Numerical experiments}\label{Sec:ESVM-DNLS:4}
In this section, we investigate the accuracy and energy-preserving behaviors of our method. For brevity, we only consider the mSP exponential integrator of order 4 (see Remark 2 in \cite{GJZ2024} for the coefficients of the prediction-correction Gauss Lawson RK methods), and the resulting scheme is abbreviated as {\bf  mSP4}. Additionally, we shall carry out a comprehensive comparison between our scheme with the existing conservative {\bf Scheme I-A} \cite{XZ-pd-2024} and {\bf CNS} \cite{GF-JAAC2021} schemes.

To quantify the numerical error, we define the $l^2$-error function and convergence order as, respectively
\begin{align*}
 e^{\tau,h}_{2}(t_{n}) = \Big(h\sum_{j=0}^{N-1} |\psi(x_j ,t_{n}) - \psi_j^{n} |^2\Big)^{1/2},\qquad {\rm Order}=\ln\left(e_{2}^{\tau_1,h}(t_n)/e_{2}^{\tau_2,h}(t_n)\right)/\ln(\tau_1/\tau_2).
\end{align*}
The residual functions on the mass, energy and momentum are defined as, respectively
\begin{align*}
e_{\mathcal W}(t_n)=\left|\mathcal W_h[{\bm\psi}^{n}]-\mathcal W_h[{\bm\psi}^{0}]\right|,\qquad \mathcal{W}_h= \mathcal M_h,\ \mathcal E_h,\  {\rm or}\  \mathcal P_h.
\end{align*}

In our computations, we set $M=5$ for the mSP method, and employ the fixed-point iteration to solve the coupled nonlinear equations of Scheme I-A and CNS, in which the convergence tolerance is set as $1.0\times 10^{-14}$. All numerical results are obtained by running MATLAB R2015b on a Win10 machine with Intel Core i7-9700 and 32 GB.

\begin{ex}[Accuracy test]\label{DNLS-exmp1} In this example, we first test the spatial and temporal accuracies of proposed numerical scheme for the wave function at $T=1$ of the DNLS equation \eqref{DNLS-equation}, and then some numerical comparisons between our schemes with the two existing conservative schemes are carried out. We note that the DNLS equation \eqref{DNLS-equation} admits an exact solution\cite{GLL-SAM2013,XZ-pd-2024}
\begin{align}\label{DNLS-two-solition1}
\psi(x,t)=g\cdot f^\ast/f^2,
\end{align} where
\begin{align*}
f&=1-\frac{3652820608655357}{2361183241434822606848}e^{4x+4}+(-\frac{1}{80}+\frac{\rm i}{8})e^{-\frac{2}{5}t+2x+2}+(\frac{1}{80}+\frac{\rm i}{8})e^{\frac{2}{5}t+2x+2}\\
&~~~~+(-\frac{5}{404}+\frac{25}{202}{\rm i})e^{2+(2-\frac{1}{5}{\rm i})x}+(\frac{5}{404}+\frac{25}{202}{\rm i})e^{2+(2+\frac{1}{5}{\rm i})x},\\
g&=\frac{{\rm i}}{50}\left[\left(\frac{25}{404}+\frac{5}{808}{\rm i}\right)e^{1+(-\frac{1}{5}-\frac{99}{100}{\rm i})t+\sigma_1}+\left(\frac{25}{404}-\frac{5}{808}{\rm i}\right)e^{1+(\frac{1}{5}-\frac{99}{100}{\rm i})t+\sigma_2}\right]e^{2x+2+\frac{99{\rm i}}{50}t}\\
&~~~~+e^{1+(-\frac{1}{5}+\frac{99}{100}{\rm i})t+\sigma_2}+e^{1+(\frac{1}{5}+\frac{99}{100}{\rm i})t+\sigma_1}
\end{align*}
with $\sigma_1 = (1-\frac{1}{10}{\rm i})x$ and $\sigma_2 = (1+\frac{1}{10}{\rm i})x$. The initial condition is always chosen as the exact solution \eqref{DNLS-two-solition1} at $t=0$ and $b=-a=50$.

\end{ex}
Table \ref{DNL-spatial-order} displays the spatial numerical error of mSP4 for the wave function $\psi(\cdot ,T=1)$. As illustrated, our method is {\em spectrally} accurate in space. Then, Tables \ref{DNL-temporal-order}-\ref{DNL-temporal-beta-order} present the temporal numerical error and the convergence order of mSP4 for the wave function $\psi(\cdot ,T=1)$ and the supplementary variables $\beta_1$, $\beta_2$ and $\beta_3$, respectively, from which one can see that (i) mSP4 is fourth-order in time; (ii) The supplementary
variables $\beta_1^n$, $\beta_2^n$ and $\beta_3^n$ can achieve fourth-order accuracy in time, which is consistent with the theoretical analysis in Remark \ref{esvm-rmk-alpha}.

\begin{table}[h]
\caption{Spatial numerical error of mSP4 for the wave function $\psi(\cdot ,T=1)$ with $\tau=10^{-5}$ in Example \ref{DNLS-exmp1}.}\label{DNL-spatial-order}
\begin{tabular}{c c c c c c }
  \hline
{} &{$h_0=100/128$}&{$h_0/2$} &{$h_0/4$}& {$h_0/8$}&{$h_0/16$}\\\hline
 {$e_{2}^{\tau,h}(T=1)$}&{3.194E-00}& {2.800E-01}&{1.603E-03} &{1.551E-08}&{2.857E-11}\\[1ex]\hline
\end{tabular}
\end{table}

\begin{table}[h]
\caption{Temporal numerical error and convergence order of mSP4 for the wave function $\psi(\cdot ,T=1)$ with $h=25/512$ in Example \ref{DNLS-exmp1}.}\label{DNL-temporal-order}
\begin{tabular}{c c c c c c}
  \hline
{} &{} &{$\tau_0=1/200$} &{$1/400$}& {$1/600$}&{$1/800$}\\
\hline
 {}  &{$e_{2}^{\tau,h}(T=1)$}&{3.502E-05}& {2.115E-06}&{4.151E-07}&{1.310E-07}\\[1ex]
  {}  &{Order}& {--}&{4.049}&{4.016} &{4.008}\\[1ex]\hline
\end{tabular}
\end{table}

\begin{table}[h]
\caption{Temporal numerical error and the convergence order of mSP4 for supplementary
variables $\beta_1^n$, $\beta_2^n$ and $\beta_3^n$ at final time $T=1$ with $h=25/512$ in Example \ref{DNLS-exmp1}.}\label{DNL-temporal-beta-order}
\begin{tabular}{ c c c c c c}
  \hline
{}  &{$\tau_0=1/300$} &{$1/400$}& {$1/500$}&{$1/600$}\\
\hline
 {$|\beta_1^n-0|$}  &{2.338E-06}& {7.283E-07}&{2.960E-07}&{1.430E-07}\\[1ex]
   {Order}& {--}&{4.054}&{4.034} &{3.989}\\[1ex]

  {$|\beta_2^n-0|$}  &{ 2.406E-06}& {7.501E-07}&{3.050E-07}&{1.474E-07}\\[1ex]
   {Order}& {--}&{4.051}&{4.033} &{3.988}\\[1ex]
  {$|\beta_3^n-0|$}  &{1.457E-08}& {4.953E-09}&{2.076E-09}&{1.019E-09}\\[1ex]
   {Order}& {--}&{3.751}&{3.896} &{3.906}\\[1ex]\hline
\end{tabular}
\end{table}

\begin{table}[h]
\caption{Numerical errors of different methods for the wave function $\psi(\cdot ,T=1)$ with different spatial and time steps in Example \ref{DNLS-exmp1}.}
\label{Tab:DNLS-1D-error}
\begin{tabular}{c c c c c c c c}\hline
{Method} &{}&{$(h_0=\frac{1}{10},\tau_0=\frac{1}{200})$}&{$(h_0/2,\tau_0/2)$} &{$(h_0/4,\tau_0/4)$} &{$(h_0/8,\tau_0/8)$}  \\
\hline
\multirow{2}{*}{mSP4}
  &{$e^{\tau,h}_{2}(T=1)$}&{3.502E-05}&{2.115E-06}&{1.310E-07}&{8.173E-09}
  \\[1ex]
  \multirow{2}{*}{CNS \cite{GF-JAAC2021}}
  &{$e^{\tau,h}_{2}(T=1)$}&{2.296E-00}&{8.140E-01}&{2.273E-01}&{5.855E-02}
  \\[1ex]
  \multirow{2}{*}{Scheme I-A \cite{XZ-pd-2024}}
  &{$e^{\tau,h}_{2}(T=1)$}&{8.170E-01}&{1.978E-01}&{4.891E-02}&{1.219E-02}
  \\[1ex]\hline
\end{tabular}
\end{table}

To compare the accuracy performance with the existing methods, Table \ref{Tab:DNLS-1D-error} shows numerical errors of different methods for the wave function $\psi(\cdot ,T=1)$ with different spatial and time steps, from which one can observe that, for a fixed spatial and time step, the mSP4 scheme produces the smallest numerical error, while the error produced by CNS is the largest.

To show the computational efficiency of the different methods for the wave function $\psi(\cdot ,T=1)$, Figure \ref{DNLS-CPU-time} displays the evolution of the CPU time
with respect to $e^{\tau,h}_{2}$ at final time $T=1$ with different spatial and time steps. It is clear to observe that for a given error, the CPU cost of mSP4 is the cheapest, while the one provided by CNS is the most expensive.

Finally, we summarize the evolution of the soliton solution, the dynamic conservation laws and the number of the Newton iteration at every time step in Figures \ref{DNLS-soliton1}-\ref{DNLS-exam1-NNI}. As illustrated, we can observe the following facts: (i) The proposed scheme can simulate the soliton well, however the profile of the soliton solution provided by CNS is wrong; (ii) The Scheme I-A can capture the amplitude and waveform, but the small disturbance emerges and it admits a large phase error; (iii) Our schemes can exactly conserve the mass, energy and momentum, however the Scheme I-A and CNS schemes can only exactly conserve the mass; (iv) The Newton
iteration method converges in less than 3 iterations. Additionally, Figure \ref{DNLS-time-spatial-steps} shows the numerical error $e^{\tau,h}_{2}$ with respect to different spatial and  time steps of different methods for the wave function $\psi(\cdot ,T=1)$, from which one can observe that, when the spatial step is insufficiently fine, the Scheme I-A and CNS schemes  will lose the accuracy, which implies that the oscillations and the large phase error of numerical solutions provided by the Scheme I-A and CNS schemes can be eliminated with a sufficiently fine spatial discretization.

\begin{figure}[H]
\centering \includegraphics[width=9.5cm,height=6.5cm]{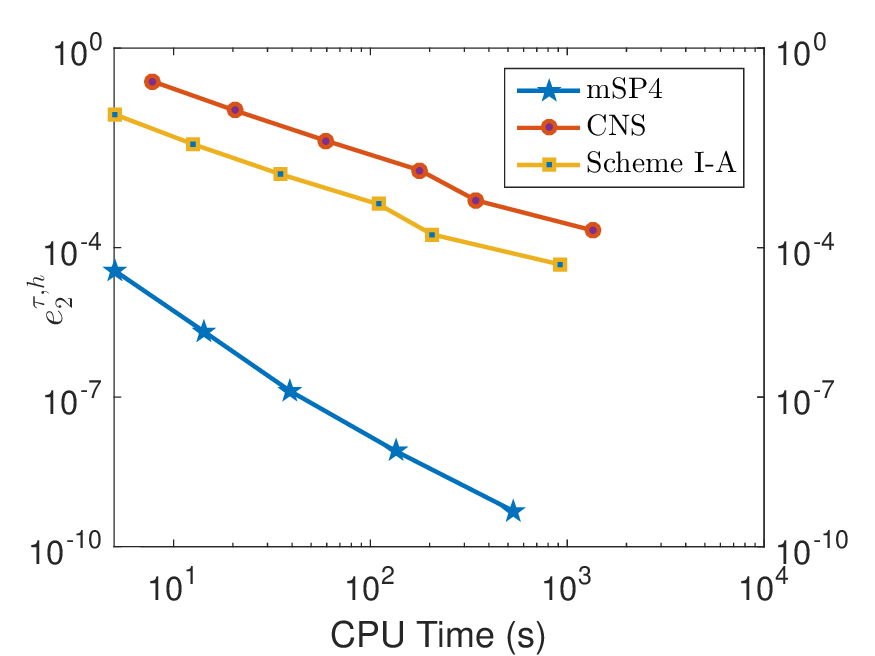}

\caption{The evolution of the CPU time
with respect to $e^{\tau,h}_{2}$ at final time $T=1$ in Example \ref{DNLS-exmp1}.}\label{DNLS-CPU-time}
\end{figure}

\begin{figure}[H]
\centering \includegraphics[width=4.0cm,height=4.0cm]{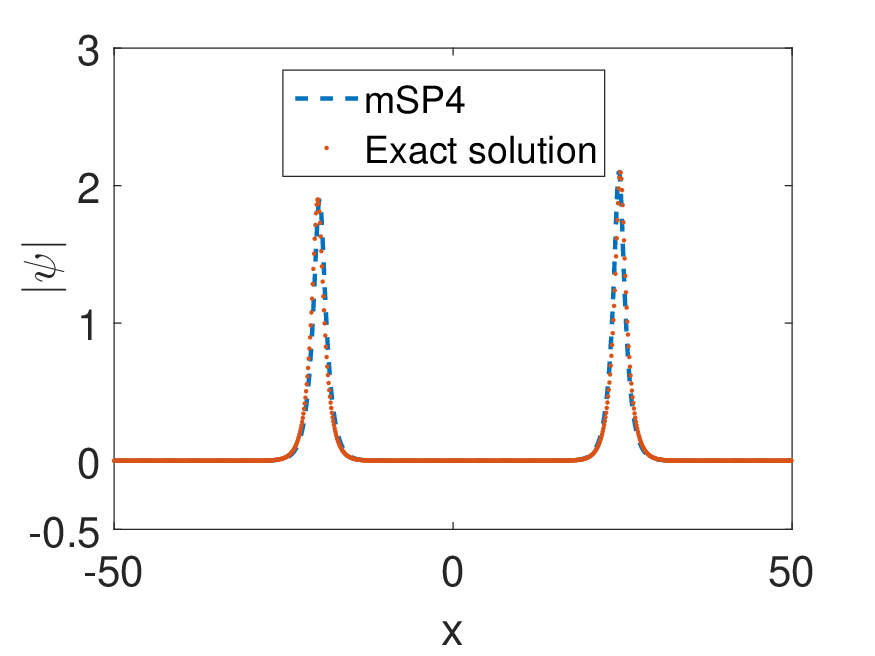}\quad
\includegraphics[width=4.0cm,height=4.0cm]{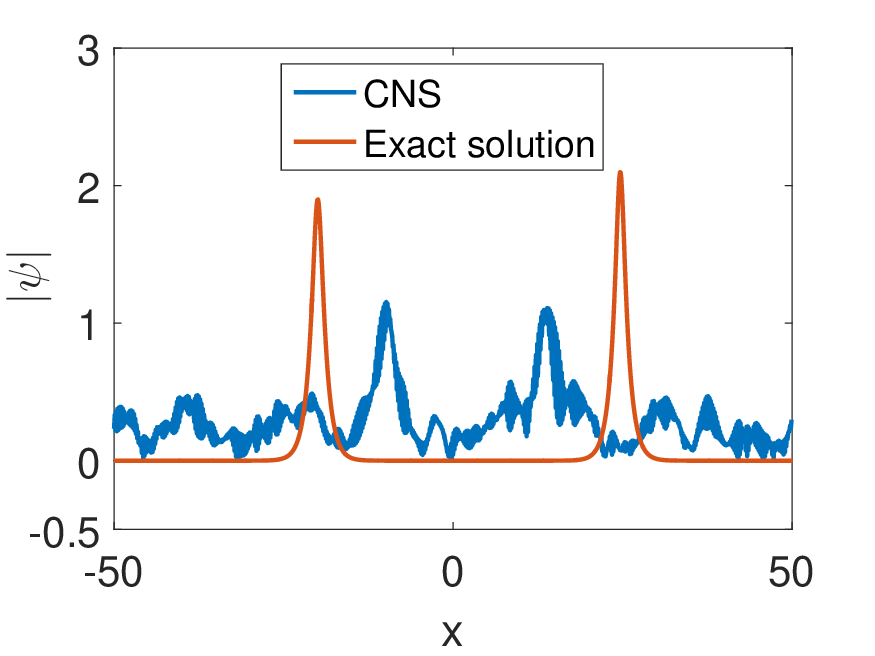}\quad
\includegraphics[width=4.0cm,height=4.0cm]{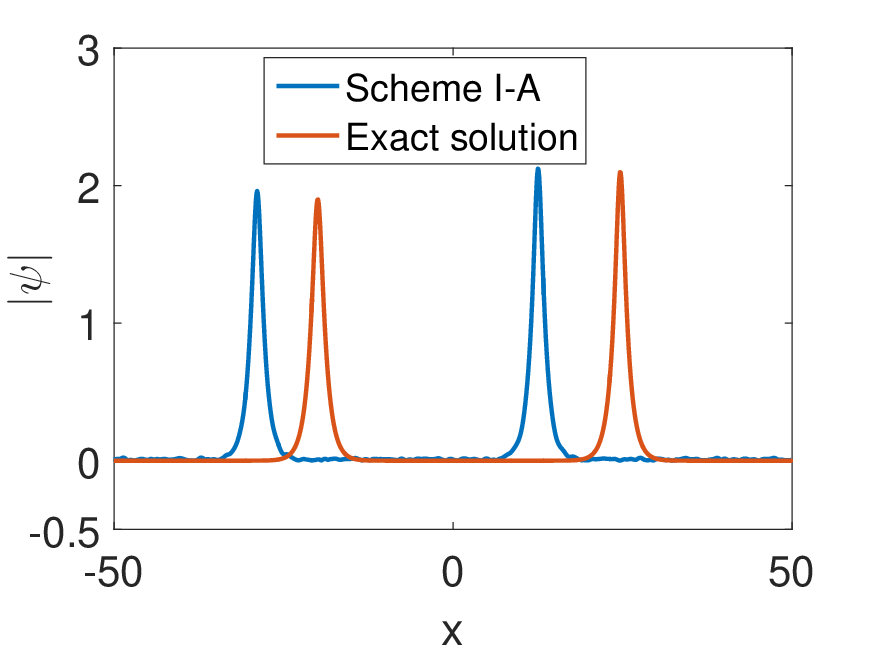}
\caption{The profile of the soliton solution for different schemes at $T=100$ with the spatial step $h=25/256$ and the time step $\tau=0.001$in Example \ref{DNLS-exmp1}.}\label{DNLS-soliton1}
\end{figure}

\begin{figure}[H]
\centering \includegraphics[width=4.0cm,height=4.0cm]{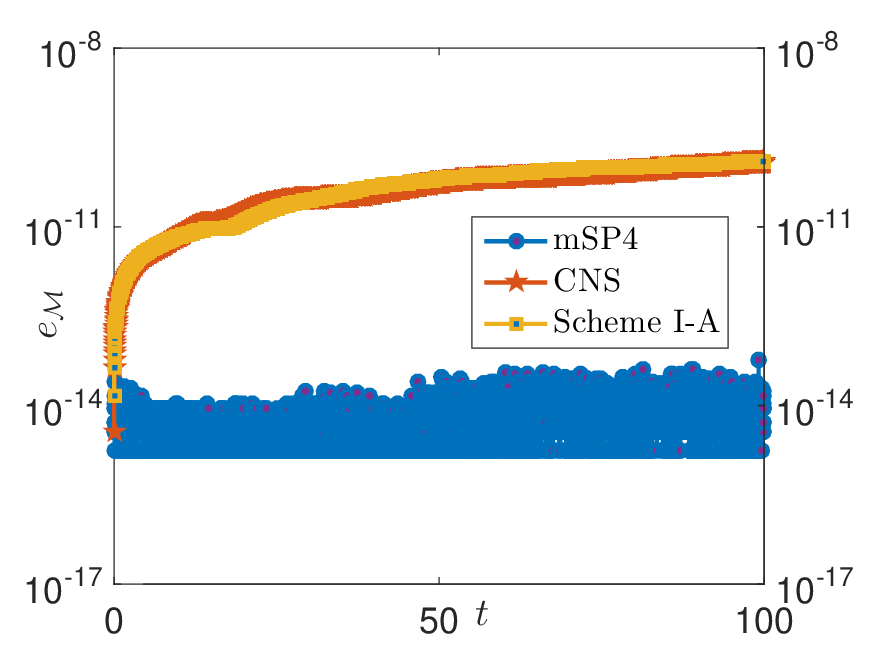}\quad
\includegraphics[width=4.0cm,height=4.0cm]{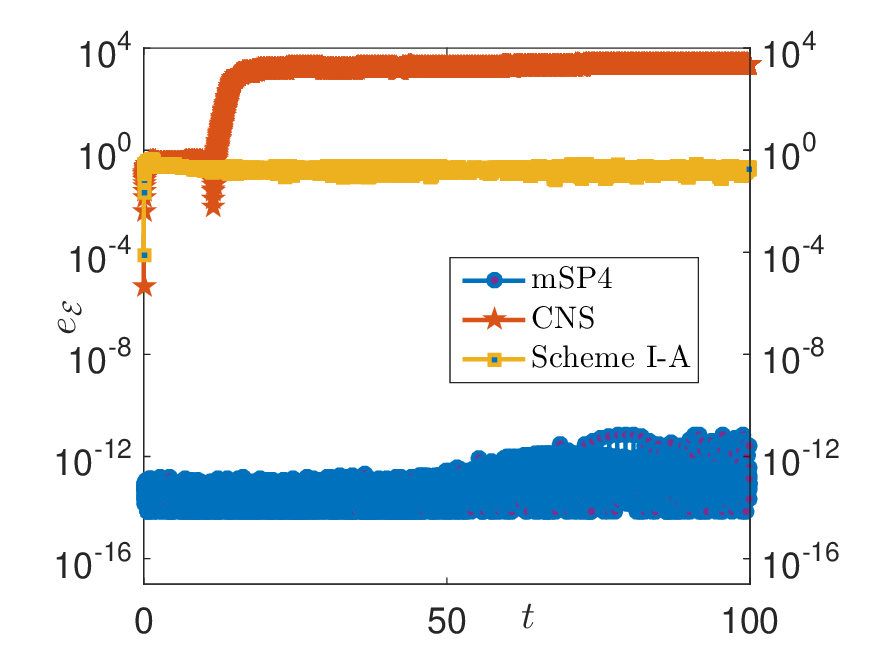}\quad
\includegraphics[width=4.0cm,height=4.0cm]{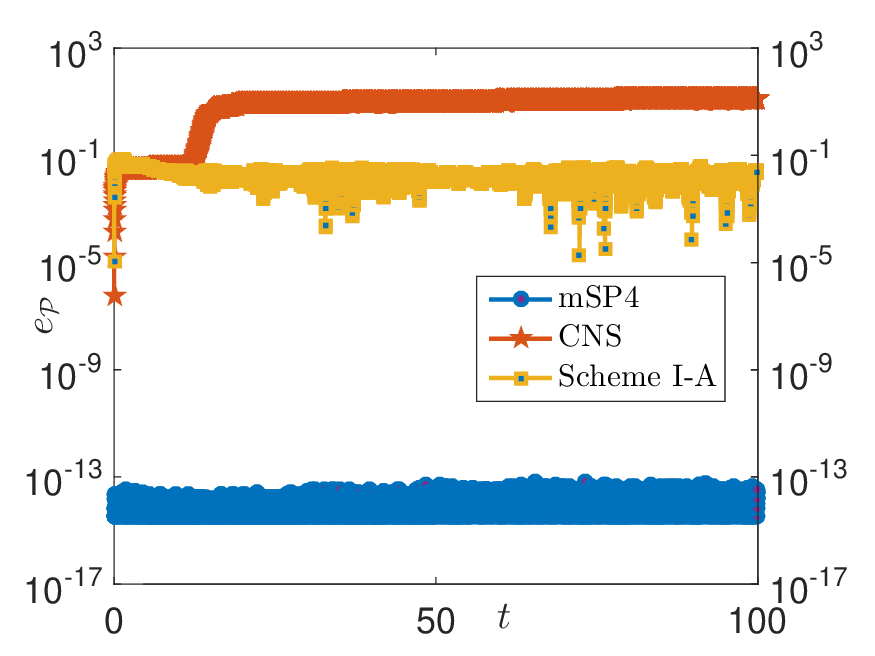}\\

\caption{The evolution of $e_{\mathcal{M}}$, $e_{\mathcal{E}}$ and $e_{\mathcal{P}}$~for different schemes over the time domain $t\in[0,100]$ with the spatial step $h=25/256$ and the time step $\tau=0.001$ in Example \ref{DNLS-exmp1}.}\label{DNLS-conservation-laws}
\end{figure}

\begin{figure}[H]
\centering \includegraphics[width=9.5cm,height=6.5cm]{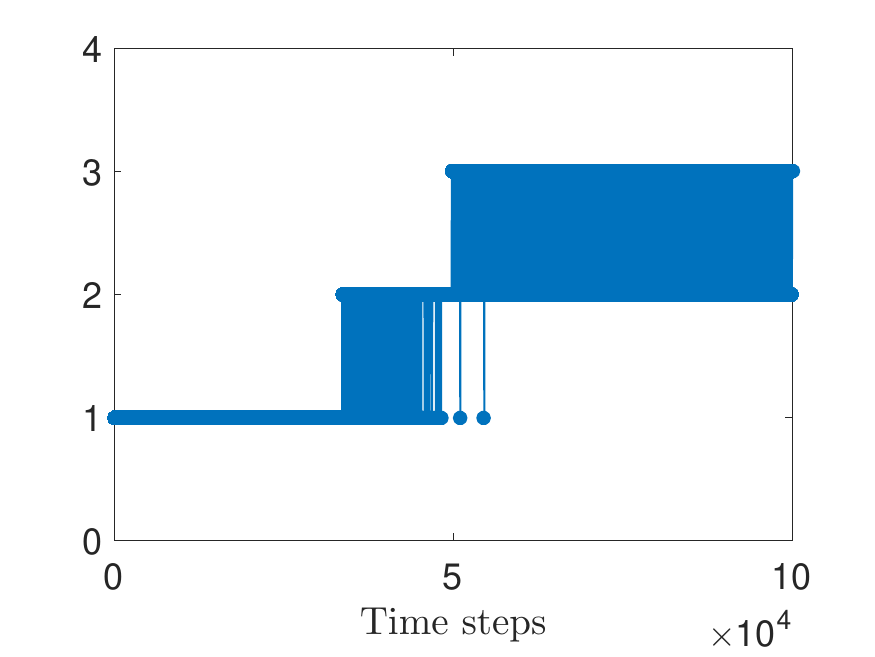}

\caption{The number of the
Newton iteration at every time
step with the spatial step $h=25/256$ and the time step $\tau=0.001$ in Example \ref{DNLS-exmp1}.}\label{DNLS-exam1-NNI}
\end{figure}

\begin{figure}[H]
\centering \includegraphics[width=4.0cm,height=4.0cm]{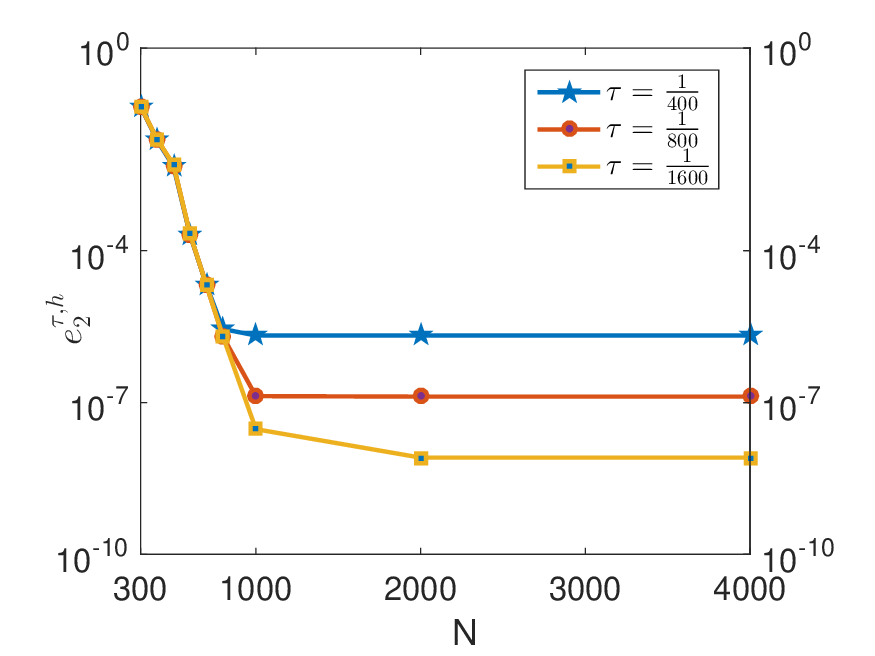}\quad
\includegraphics[width=4.0cm,height=4.0cm]{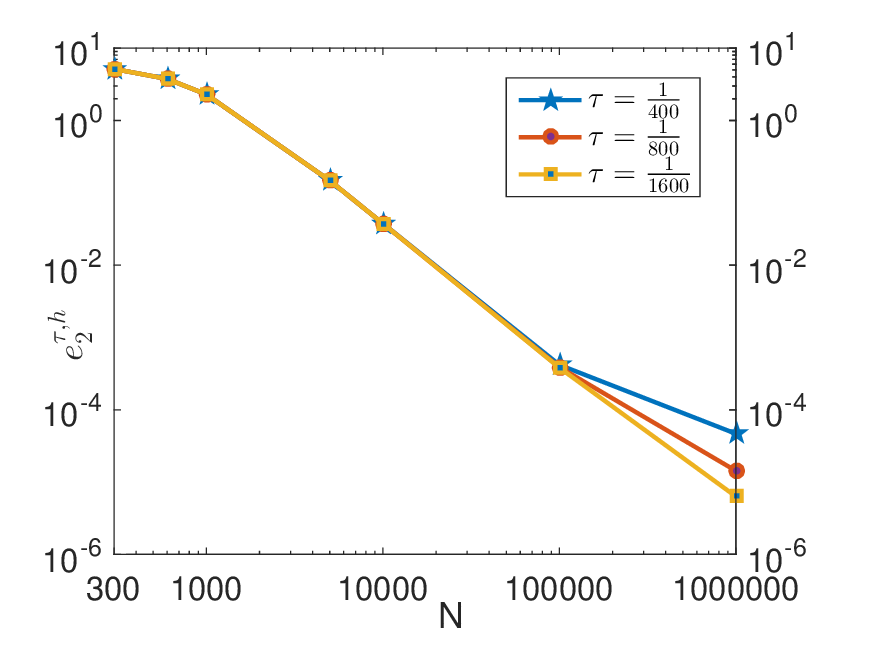}\quad
\includegraphics[width=4.0cm,height=4.0cm]{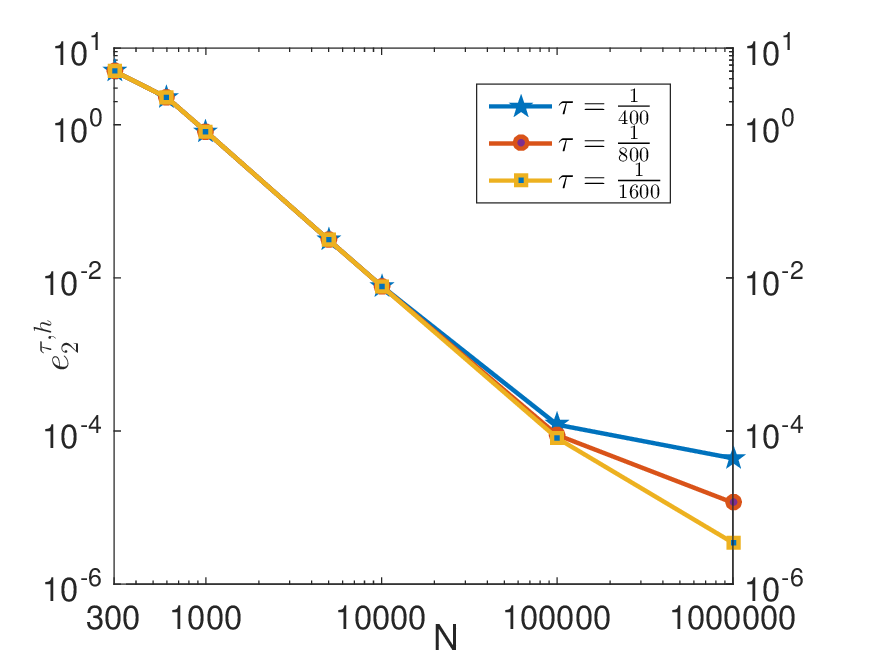}\\

\caption{The numerical error $e^{\tau,h}_{2}$ with respect to different spatial and time steps provided by mSP4 (left), CNS (middle) and Scheme I-A (right) for the wave function $\psi(\cdot ,T=1)$ in Example \ref{DNLS-exmp1}.}\label{DNLS-time-spatial-steps}
\end{figure}

\begin{ex}[Two-soliton solutions]\label{DNLS-exmp2} In this example, the proposed scheme will be employed to simulate the evolution of two-soliton solutions by choosing the following initial condition\cite{KN-JMP-1978}
\begin{align*}
\psi(x,t)|_{t=-10}=g\cdot f^\ast/f^2,
\end{align*} where
\begin{align*}
f&=1+\frac{{\rm i}k_1}{2(k_1+k_1^\ast)^2}e^{\eta_1+\eta_1^\ast}+\frac{{\rm i}k_1}{2(k_1+k_2^\ast)^2}e^{\eta_1+\eta_2^\ast}+\frac{{\rm i}k_2}{2(k_2+k_1^\ast)^2}e^{\eta_2+\eta_1^\ast}+\frac{{\rm i}k_2}{2(k_2+k_2^\ast)^2}e^{\eta_2+\eta_2^\ast}\\
&+\frac{-k_1k_2(k_1-k_2)^2(k_1^\ast-k_2^\ast)^2}{4(k_1+k_1^\ast)^2(k_1+k_2^\ast)^2(k_2+k_1^\ast)^2(k_2+k_2^\ast)^2}e^{\eta_1+\eta_1^\ast+\eta_2+\eta_2^\ast},\\
g&=e^{\eta_1}+e^{\eta_2}+\frac{-k_1^\ast(k_1-k_2)^2}{2(k_1+k_1^\ast)^2(k_2+k_1^\ast)^2}e^{\eta_1+\eta_1^\ast+\eta_2}
+\frac{-k_2^\ast(k_1-k_2)^2}{2(k_1+k_2^\ast)^2(k_2+k_2^\ast)^2}e^{\eta_1+\eta_2+\eta_2^\ast}
\end{align*}
with $\eta_i=k_ix+\omega_it+1,\ i=1,2$ and $k_1=1+0.3{\rm i},\ k_2=1-0.3{\rm i},\ \omega_i={\rm i}k_i^2,\ i=1,2$. The computational domain is chosen as $\Omega=[-50,50]$.

\end{ex}

Figure \ref{DNLS-soliton2} presents the numerical evolution of the physical mechanism involved in the mutual interaction of two
solitons produced by the mSP4 (left), CNS (middle) and Scheme I-A (right) schemes, respectively, from which one can observe that mSP4 can well capture the shape of the solitons, while the results provided by Scheme I-A and CNS are wrong. We note that, similar to Example \ref{DNLS-exmp1}, the shape of the solitons can be well captured with a sufficiently fine spatial discretization ($h=25/512$ and $h=25/1024$ for Scheme I-A and CNS, respectively in this example). In Figure \ref{DNLS-conservation-laws2}, we investigate the residuals on the mass, energy and momentum provided by the three conservative schemes, which implies that our scheme can exactly conserve the three conservation laws, however Scheme I-A and CNS can only exactly preserve the mass. Collectively, these findings demonstrate that our scheme exhibits superior numerical performance compared to the existing Scheme I-A and CNS schemes. Finally, we investigate the corresponding number of the Newton iteration at every time step in Figure \ref{DNLS-exam2-NNI}, which implies that the Newton
iteration method converges in less than 3 iterations.

\begin{figure}[H]
\centering\includegraphics[width=4.0cm,height=4.0cm]{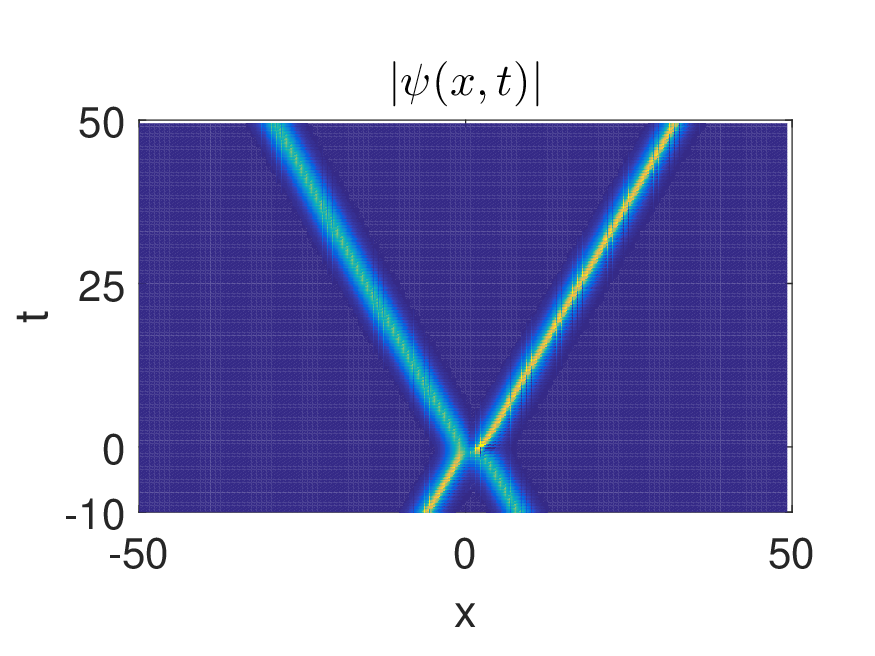}\quad
	\includegraphics[width=4.0cm,height=4.0cm]{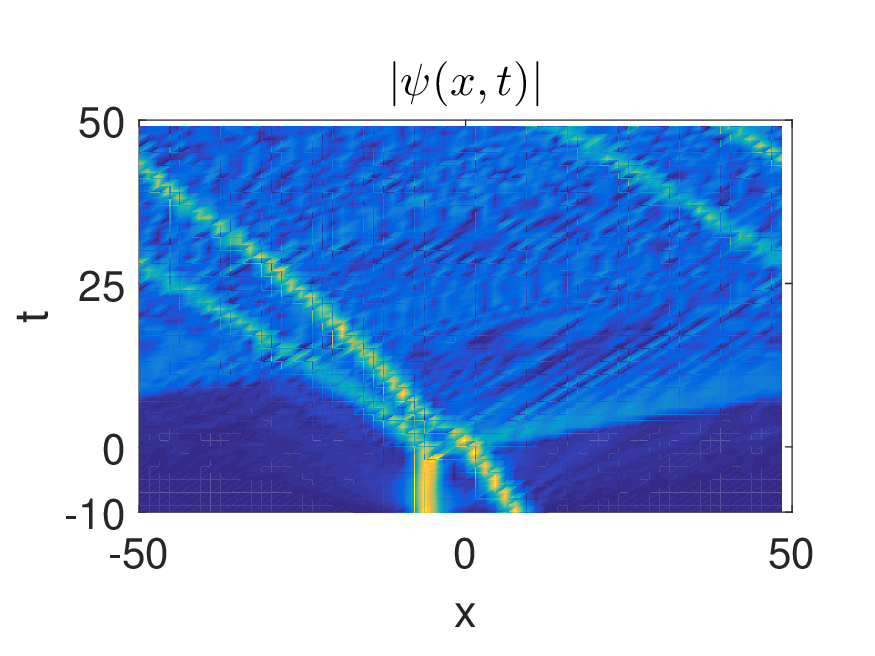}\quad
\centering\includegraphics[width=4.0cm,height=4.0cm]{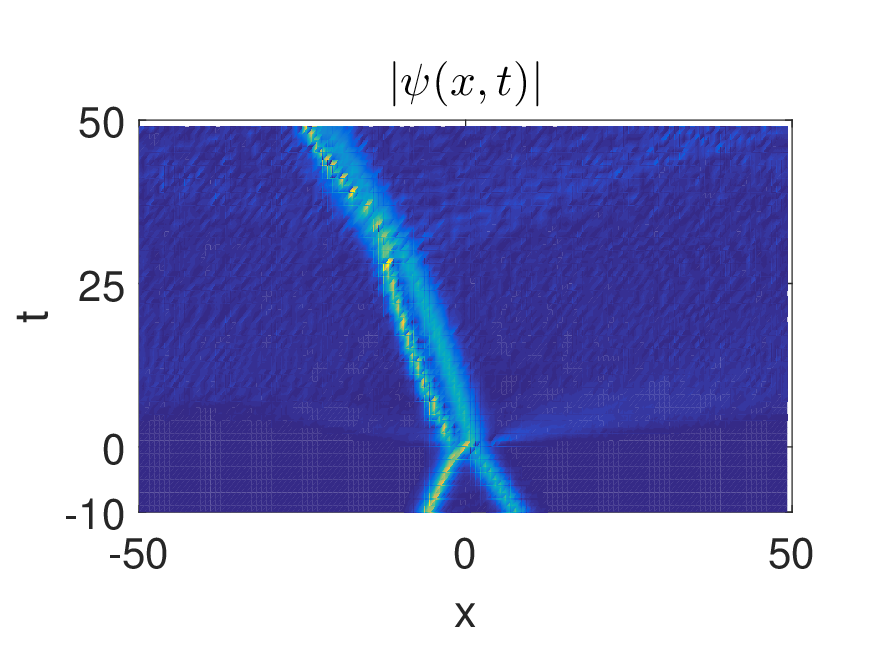}
\caption{The numerical evolution of two soliton solutions produced by mSP4 (left), CNS (middle) and Scheme I-A (right) with the spatial step $h=25/256$ and the time step $\tau=0.001$ in Example \ref{DNLS-exmp2}.}\label{DNLS-soliton2}
\end{figure}

\begin{figure}[H]
\centering \includegraphics[width=4.0cm,height=4.0cm]{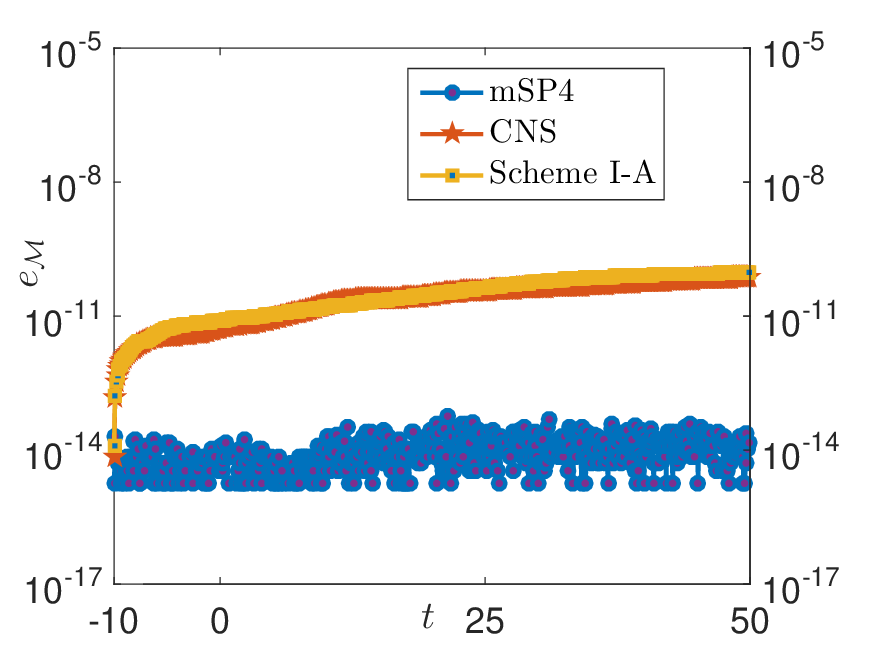}\quad
\includegraphics[width=4.0cm,height=4.0cm]{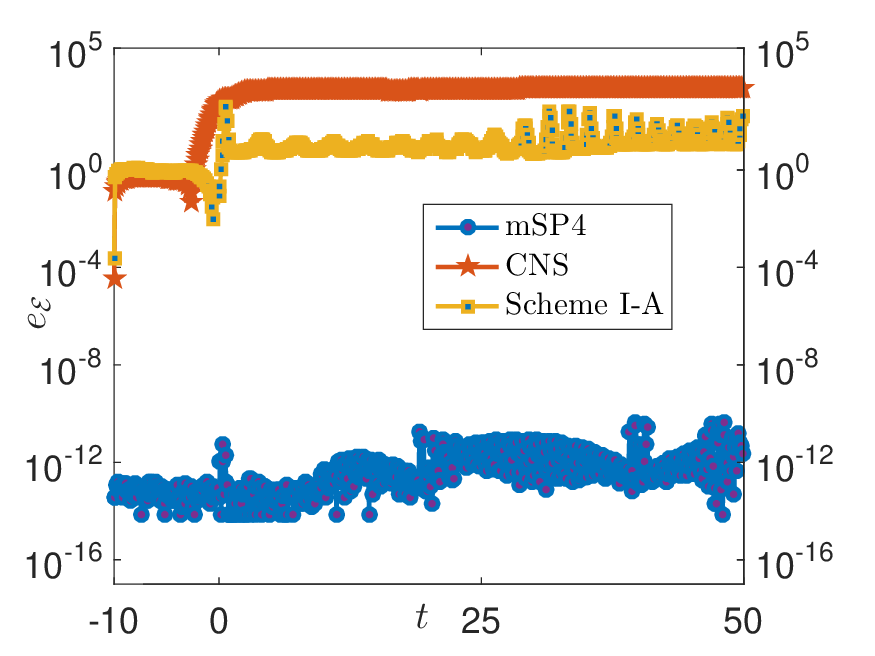}\quad
\includegraphics[width=4.0cm,height=4.0cm]{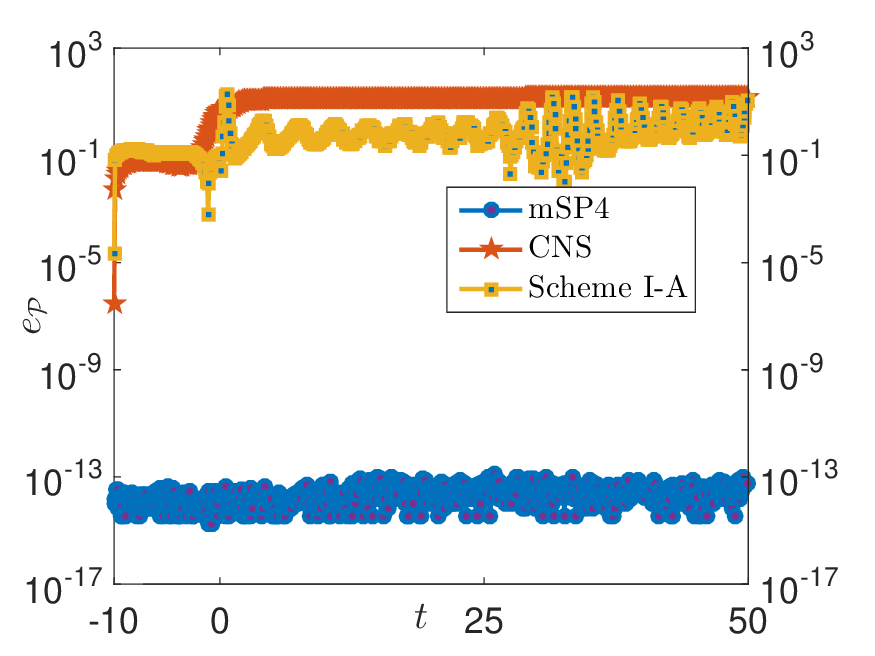}\\

\caption{The evolution of $e_{\mathcal{M}}$, $e_{\mathcal{E}}$ and $e_{\mathcal{P}}$~for different schemes over the time domain $t\in[-10,50]$ with the spatial step $h=25/256$ and the time step $\tau=0.001$ in Example \ref{DNLS-exmp2}.}\label{DNLS-conservation-laws2}
\end{figure}

\begin{figure}[H]
\centering \includegraphics[width=6.5cm,height=5.5cm]{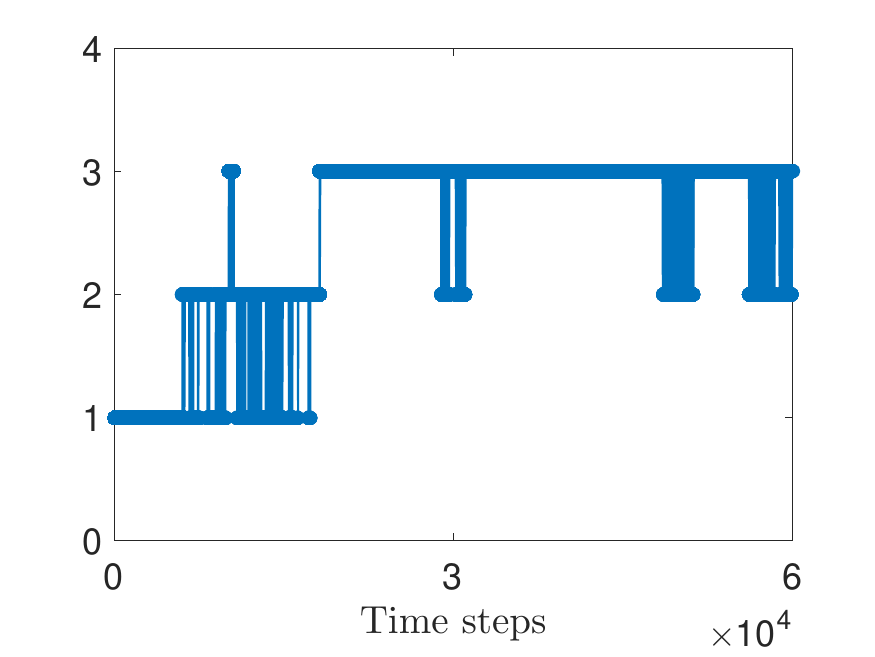}
\caption{The number of the
Newton iteration at every time
step  with the spatial step $h=25/256$ and the time step $\tau=0.001$ in Example \ref{DNLS-exmp2}}.\label{DNLS-exam2-NNI}
\end{figure}

\section{Concluding remarks}\label{Sec:ESVM-DNLS:5}

In this paper, we propose a class of high-order multi-structures-preserving exponential integrators for the DNLS equation \eqref{DNLS-equation} based on the idea of the ESVM approach. The proposed scheme is efficient and easy to implement and can simultaneously conserve the mass, energy and momentum. Numerical results are presented to confirm the accuracy and long time evolution of dynamical conservation laws.
Compared with the existing conservative schemes, our scheme has an advantage on the accuracy, robustness and preservation of dynamic conservation laws. Moreover,
 with the aid of the gauge transformation,
\begin{align*}
u(x,t)=\psi(x,t)e^{\frac{{\rm i}}{2}\int_{-\infty}^x|\psi(y,t)|^2dy},
\end{align*}
the DNLS equation  \eqref{DNLS-equation} reduces to the second-type derivative nonlinear Schr\"odinger equation or Chen-Lee-Liu equation \cite{CLL-PS-1979}
\begin{align*}
 &\text{i}\frac{\partial}{\partial t}u(x,t)=-\partial_{xx}u(x,t)-{\rm i}\big(|u(x,t)|^2\partial_xu(x,t)\big),\quad  t>0,
\end{align*}
which satisfies the mass, energy and momentum similar to the DNLS equation \eqref{DNLS-equation}\cite{LLS-IJNSNS-2018,XZM-AML-2025,XZ-pd-2026}.
We note that the proposed strategies can be extended to construct multi-structures-preserving exponential integrators for the second-type DNLS equation.

\backmatter

\bmhead{Acknowledgements}

The authors would like to express sincere gratitude to the referees for their insightful comments and suggestions.

\section*{Declarations}
\begin{itemize}
\item Funding This work is supported by the National Natural Science Foundation of China (Grant No. 62262068) and Yunnan Fundamental Research Project (Grant Nos. 202401AT070283).\vspace{0.1cm}
\item Conflict of interest/Competing interests The authors declare no competing interests.\vspace{0.1cm}
\item Data availability No datasets were generated or analysed during the current study. \vspace{0.1cm}
\item Author contribution Liping Wu: designed the numerical algorithm and wrote the
main manuscript tex. Li Yang: validation, conducted numerical experiments. Chaolong Jiang: designed the
numerical algorithm, conducted numerical experiments and wrote the main manuscript text.
\end{itemize}

\end{document}